\documentclass[10pt]{amsart}
\usepackage{amsthm,amsmath,amssymb,amscd,graphics,enumerate, stmaryrd,xspace,verbatim, epic, eepic,url}

\usepackage[usenames,dvipsnames]{color}

\usepackage[colorlinks=true]{hyperref}

\usepackage[active]{srcltx}
\usepackage[all]{xypic}
\SelectTips{cm}{}
{
   \newtheorem{theorem}[subsubsection]{Theorem}
      \newtheorem*{theorem*}{Theorem}
   \newtheorem{proposition}[subsubsection]{Proposition}
   \newtheorem{hypothesis}[subsubsection]{Hypothetical Statement}
   
   \newtheorem{lemma}[subsubsection]{Lemma}

   \newtheorem*{conjecture*}{Conjecture}
   
   \newtheorem{assumption}[subsubsection]{Assumption}
}
{\theoremstyle{definition}
          \newtheorem*{exercise*}{Exercise}

   \newtheorem*{example*}{Example}
   \newtheorem{definition}[subsubsection]{Definition}
   
   \newtheorem*{definition*}{Definition}
   
   \newtheorem{remark}[subsubsection]{Remark}

}
\newcommand{\RR}{{\mathbb{R}}}

\newcommand{\FF}{{\mathbb{F}}}
\newcommand{\CC}{{\mathbb{C}}}
\newcommand{\QQ}{{\mathbb{Q}}}
\newcommand{\NN}{{\mathbb{N}}}

\newcommand{\PP}{{\mathbb{P}}}
\newcommand{\ZZ}{{\mathbb{Z}}}

\newcommand{\GG}{{\mathbb{G}}}

\renewcommand{\AA}{{\mathbb{A}}}
\newcommand{\HH}{{\mathbb{H}}}

\newcommand{\by}{{\mathbf{y}}}
\def\hatcO{{\widehat\cO}}

\newcommand{\fX}{{\mathfrak{X}}}
\newcommand{\fY}{{\mathfrak{Y}}}
\newcommand{\fSp}{{\mathfrak{Sp}}}
\newcommand{\fSt}{{\mathfrak{St}}}

\newcommand{\cC}{{\mathcal C}}

\newcommand{\cE}{{\mathcal E}}
\newcommand{\cF}{{\mathcal F}}
\newcommand{\cG}{{\mathcal G}}
\renewcommand{\cH}{{\mathcal H}}
\newcommand{\cI}{{\mathcal I}}

\newcommand{\cM}{{\mathcal M}}

\newcommand{\cO}{{\mathcal O}}
\newcommand{\cP}{{\mathcal P}}

\newcommand{\cV}{{\mathcal V}}

\newcommand{\cX}{{\mathcal X}}
\newcommand{\cY}{{\mathcal Y}}
\newcommand{\cZ}{{\mathcal Z}}

\newcommand{\oJ}{{\overline{J}}}
\newcommand{\oE}{{\overline{E}}}

\newcommand{\oM}{{\overline{M}}}

\newcommand{\oh}{{\overline{h}}}
\newcommand{\an}{{\rm an}}

\def\<{\langle}
\def\>{\rangle}

\newcommand{\Spec}{\operatorname{Spec}}
\newcommand{\Spf}{\operatorname{Spf}}
\newcommand{\Proj}{\operatorname{Proj}}

\newcommand{\Hom}{{\operatorname{Hom}}}
\newcommand{\cHom}{{{\cH}om}}

\newcommand{\Coh}{{\operatorname{Coh}}}

\newcommand{\Ker}{{\operatorname{Ker}}}
\newcommand{\Aut}{{\operatorname{Aut}}}

\newcommand{\Sym}{{\operatorname{Sym}}}

\newcommand{\chara}{\operatorname{char}}

\newcommand{\das}{\dashrightarrow}

\newcommand{\bfA}{{\mathbf A}}

\newcommand{\oD}{{\overline{D}}}

\newcommand{\sst}{{\operatorname{sst}}}

\newcommand{\un}{{\operatorname{un}}}
\newcommand{\sing}{{\operatorname{sing}}}

\newcommand{\Bl}{{\operatorname{Bl_{rs}}}}
\newcommand{\TorBl}{{\operatorname{TorBl_{rs}}}}
\newcommand{\Fact}{{\operatorname{Fact_{rs}}}}
\newcommand{\TorFact}{{\operatorname{TorFact_{rs}}}}
\newcommand{\Cob}{{\operatorname{Cob_{rs}}}}

\def\:{{\colon}}
\def\.{{,\dots,}}

\newcommand{\double}{\genfrac..{0pt}1
{\raise -1pt\hbox{$\scriptstyle\longrightarrow$}}{\raise 3pt\hbox
{$\scriptstyle\longrightarrow$}}}

\renewcommand{\setminus}{\smallsetminus}

\def\tototi{\mathbin{\mathop{\otimes}\limits^{\raise-1pt\hbox
{$\scriptscriptstyle {\rm L}$}}}}

\def\indlim{\mathop{\vrule width0pt height7pt depth
4pt\smash{\lim\limits_{\raise 1pt\hbox to 14.5pt
{\rightarrowfill}}}}}
\def\projlim{\mathop{\vrule width0pt height7pt depth
4pt\smash{\lim\limits_{\raise 1pt\hbox to 14.5pt
{\leftarrowfill}}}}}

\newcommand\displaceamount{3pt}

\newcommand{\doubledown}{\ar@<\displaceamount>[d]\ar@<-\displaceamount>[d]}

\newcommand{\doubleup}{\ar@<\displaceamount>[u]\ar@<-\displaceamount>[u]}

\newcommand{\doubleright}{\ar@<\displaceamount>[r]\ar@<-\displaceamount>[r]}

\newcommand{\tor}{{\operatorname{tor}}}
\newcommand{\res}{{\operatorname{res}}}
\newcommand{\can}{{\operatorname{can}}}

\def\into{\hookrightarrow}
\def\onto{\twoheadrightarrow}

\def\gp{\text{gp}}

\def\GGm{{\mathbb G}_m}

\def\tilJ{{\widetilde J}}
\def\toisom{\xrightarrow{{_\sim}}}

\def\hatB{{\widehat B}}

\def\bfD{{\mathbf{D}}}

\def\hatphi{{\widehat\phi}}

\def\reg{{\operatorname{reg}}}

\begin{document}
\title[Factorization of birational maps for qe schemes]{Functorial factorization of birational maps for qe schemes in characteristic 0}

\author[D. Abramovich]{Dan Abramovich}
\address{Department of Mathematics, Box 1917, Brown University,
Providence, RI, 02912, U.S.A}
\email{abrmovic@math.brown.edu}
\author[M. Temkin]{Michael Temkin}
\address{Einstein Institute of Mathematics\\
               The Hebrew University of Jerusalem\\
                Giv'at Ram, Jerusalem, 91904, Israel}
\email{temkin@math.huji.ac.il}
\thanks{This research is supported by  BSF grant 2010255}

\date{\today}
%\large

%\begin{abstract}
%\end{abstract}

\begin{abstract} We prove functorial weak factorization of projective birational morphisms of regular quasi-excellent schemes in characteristic 0 broadly based on the existing line of proof for varieties. From this general functorial statement we deduce factorization results for algebraic stacks, formal schemes, complex analytic germs, Berkovich analytic and rigid analytic spaces.
\end{abstract}
\maketitle
\setcounter{tocdepth}{1}

\tableofcontents

 \section{Introduction}

%\subsection{Notation}\hfill

%\begin{tabular}{lll}
%{\bf typeset} & {\bf notation} &  {\bf term} \\
%\verb.\Bl.&$\Bl$ & category of blowups, regular surjective arrows \\
%\verb.\Fact.&$\Fact$ & category of factorizations \\
%\verb.\TorBl.&$\TorBl$ & category of factorizations \\
%&$Z_i$ & center of blowup of $ \varphi_i$ or inverse \\
%&$J_i$ & ideal of blowup of $V_i \to X_2$ \\
%\end{tabular}

\subsection{}
%\Dan{\emph{To do list}
%\begin{enumerate}
%\item Finish showing that other categories follow from qe schemes
%\item State and prove stacks result
%\item set up so that the mixed characteristic case will follow once the appropriate resolution results hold.
%\item normal crossings data
%\item Make sure regularity, characteristic, noetherianity are assumed when needed!
%\end{enumerate}
%}
The class of qe schemes (originally ``quasi excellent schemes") is the natural class of schemes on which problems around resolution of singularities are of interest. They can also be used as a bridge for studying the same type of problems in other geometric categories, see \cite[Section 5]{Temkin}. In this paper we address the problem of functorial factorization of birational morphisms between regular qe schemes of characteristic 0 into blowings up and down of regular schemes along regular centers. We rely on general foundations developed in \cite{ATLuna, AT1} and the approach for varieties of \cite{W-Cobordism, AKMW}. As a consequence of both this generality of qe schemes and of functoriality, we are able to deduce factorization of birational or bimeromorphic morphisms in other geometric categories of interest.

\subsection{Blowings up and weak factorizations}\label{Sec:def-factor} We start with a morphism of noetherian qe regular schemes $\phi\:X_1 \to X_2$ given as the blowing up of a coherent sheaf of ideals $I$ on the qe scheme $X_2$. In addition, we provide $\phi$ with a {\em boundary} $(D_1,D_2)$, where each $D_i$ is a normal crossings divisor in $X_i$ and $D_1:=\phi^{-1}D_2$. Let $U=X_2\setminus(D_2\cup V(I))$ be the maximal open subscheme of $X_2$ such that $I$ is the unit ideal on $U$ and the boundary is disjoint from $U$. The restriction of $\phi$ on $U$ is the trivial blowing up (i.e. the blowing up of the empty center), in particular, we canonically have an isomorphism $\phi^{-1}U \to U$. We often keep the ideal $I$ implicit in the notation, even though it determines $\phi$ (but see Section \ref{Sec:blowup-projective} for a construction in the reverse direction). The reader may wish to focus on the following two cases of interest: (i) $D_2 = \emptyset$; (ii) $V(I)\subseteq D_2$.

% \subsection{Weak factorizations}
 A {\em weak factorization} of a blowing up $\phi\:X_1 \to X_2$ is a diagram of regular qe schemes
 %\Michael{Another option for the notation is $X'=X_l\dashrightarrow\ldots\dashrightarrow X_1\dashrightarrow X_0=X$.}
 $$\xymatrix{X_1 = V_0\ar@{-->}[r]^(.6){\varphi_1}& V_1\ar@{-->}[r]^{\varphi_2}&\ldots \ar@{-->}[r]^{\varphi_{l-1}} &V_{l-1} \ar@{-->}[r]^(.4){\varphi_l}&  V_l = X_2}$$
along with regular schemes $Z_i$ for $i=1,\ldots,l$  and ideal sheaves $J_i$ for $i=1,\ldots,(l-1)$  satisfying the following conditions:
 \begin{enumerate}
 \item $\phi = \varphi_l\circ \varphi_{l-1}\circ \cdots \circ \varphi_2\circ \varphi_1$.
% \item  $V_i$ are nonsingular qe schemes.
 \item The maps $V_i \das X_2$ are morphisms; these maps as well as $\varphi_i$ induce isomorphisms on $U$.
 \item For every $i = 1,\ldots,l$ either $\varphi_i\: V_{i-1} \das V_i$ or $\varphi_i^{-1}\: V_{i} \das V_{i-1}$ is a morphism given as the blowing up of $Z_i$, which is respectively a subscheme  of $V_i$ or $V_{i-1}$ disjoint from $U$.
 \item The inverse image $D_{V_i} \subset V_i$ of $D_2 \subset X_2$ is a normal crossings divisor and $Z_i$ has normal crossings with $D_{V_i}$.
  \item For every $i = 1,\ldots,(l-1)$, the morphism $V_i \to X_2$ is given as the blowing up of the corresponding  coherent ideal sheaf $J_i$ on $X_2$, which is the unit ideal on $U$.
 \end{enumerate}

 To include $V_0 \to X_2$, we define $J_0 = I$. The ideals $J_i$ are a convenient way to encode functoriality, especially when we later pass to other geometric categories.

 These conditions are the same as (1)--(5)  in \cite[Theorem 0.3.1]{AKMW}, except that here the centers of blowing up and ideal sheaves are specified. Condition (2) is formulated for convenience; it is a consequence of (3) and (5).
Note that here, as in  \cite[Theorem 0.3.1]{AKMW}, the centers are not assumed irreducible, in contrast with \cite[Theorem 0.1.1]{AKMW}. With these condition, the most basic form of our main theorem is as follows:

\begin{theorem}[Weak factorization]\label{Th:factor-simple}
Every birational blowing up  $\phi\:X_1 \to X_2$ of  a noetherian qe regular $\QQ$-scheme has a weak factorization $X_1 = V_0\das V_1\das\dots \das V_{l-1} \das V_l = X_2$.
\end{theorem}

The adjective ``weak" serves to indicate that blowings up and down may alternate arbitrarily among the maps $\varphi_i$, as opposed to a {\em strong factorization}, where one has a sequence of blowings up followed by a sequence of blowings down. We note that at present strong factorization is not known even for toric threefolds.

Theorem \ref{Th:factor-simple} generalizes  \cite[0.0.1]{W-Factor} and \cite[Theorem 0.1.1]{AKMW}, where the case of varieties is considered. But we wish to prove a more precise theorem.

\subsection{Functorial weak factorization}

The class of data $(X_2,I,D_2)$, namely morphisms $\phi\:X_1 \to X_2$  of noetherian qe regular schemes given as blowings up of ideals $I$, with divisor $D_2$ as in   Section \ref{Sec:def-factor},  can be made into {\em the regular surjective category of blowings up}, denoted $\Bl$, by defining arrows as follows:
\begin{definition}\label{Def:Bl} An arrow from the blowing up $\phi'\:X'_1=Bl_{I'}(X_2') \to X'_2$ to $\phi:X_1=Bl_{I}(X_2) \to X_2$ is a  regular and surjective morphism $g: X'_2\to X_2$  such that $g^* I = I'$ and $g^{-1}D_2=D'_2$. In particular, $g$ induces a canonical isomorphism $X_1'\to X_1 \times_{X_2} X_2'$ and $D'_1$ is the preimage of $D_1$ under $X'_1\to X_1$.
\end{definition}

Similarly, weak factorizations can be made into the regular surjective category  of weak factorizations, denoted $\Fact$, by defining arrows as follows:

 \begin{definition}\label{Def:factcat}  A morphism in $\Fact$ from a weak factorization  $$X'_1 = V'_0\das V'_1\das\dots \das V'_{l-1} \das V'_l = X'_2$$  of $\phi'\:X'_1 \to X'_2$, with centers $Z_i'$ and ideals $J_i'$ to a weak factorization $$X_1 = V_0\das V_1\das\dots \das V_{l-1} \das V_l = X_2$$ of $\phi\:X_1 \to X_2$, with centers $Z_i$ and ideals $J_i$ consists of a regular surjective morphism  $g\: X'_2\to X_2$ such that $g^*I = I'$, $g^*J_i= J_i'$ inducing $g_i\: V_i'\to V_i$, such that $Z_i' = g_i^{-1} Z_i$ or $g_{i-1}^{-1} Z_i$ as appropriate. In particular $\varphi_i \circ g_{i-1} = g_i \circ \varphi_i$ and $g_i^{-1}D_{V_i}=D_{V'_i}$.
  \end{definition}

%\begin{remark}\label{Rem:ncd-strict} In some applications, one is given an open set $U\subset X_2$ on which $I$ is the unit ideal, such that the closed subschemes $X_i\setminus U$ are normal crossings divisors.
%%; moreover, in \cite[Theorem 3.8 and Remark 3.9]{Borisov-Libgober} one further assume these schemes are $G$-strict with respect to a given action of a finite group.
%There is an evident category of blowings up satisfying the condition that $X_i\setminus U$ are normal crossings divisors.
%%, and similarly with the $G$-strictness assumption.
%Also there is an evident category of weak factorizations where  the divisors $E_i=V_i\setminus U$ are normal crossings and the center $Z_i$ of blowing up in $V_i$ or $V_{i-1}$ has normal crossings with $E_i$ or $E_{i-1}$ as appropriate.
%%, as well as a similar category where $E_i$ are $G$-strict.
%Rather than introduce notations for  these categories, we will simply say in our statements and arguments below that  the conditions are satisfied for a functorial factorization, keeping the relevant categories implicit.
%\end{remark}

Note that given a factorization of $\phi$, any morphism from a factorization of $\phi'$ is uniquely determined by $g\: X_2' \to X_2$.

If we wish to restrict to schemes in a given characteristic $p$ we denote the categories $\Bl({\chara = p})$  and $\Fact({\chara = p})$ respectively. If we wish to restrict the dimension we write $\Bl({\chara = p,\dim \leq d})$  and $\Fact({\chara = p,\dim\leq d})$.

There is an evident forgetful functor $\Fact \to \Bl$ taking a weak factorization $X_1 = V_0\das V_1\das\dots \das V_{l-1} \das V_l = X_2$ to its composition $\phi\:X_1 \to X_2$.   The weak factorization theorem provides a section, when strong resolution of singularities holds:

 \begin{theorem}\label{Th:main}
 \begin{enumerate}
 \item {\sc Functorial weak factorization:} There is a functor $$\Bl({\chara = 0}) \to \Fact({\chara = 0})$$ assigning to a  blowing up $\phi\:X_1 \to X_2$ in characteristic 0 a weak factorization $$X_1 = V_0\das V_1\das\dots \das V_{l-1} \das V_l = X_2,$$ so that the composite  $\Bl({\chara = 0}) \to \Fact({\chara = 0})\to \Bl({\chara = 0})$ is the identity.

%\item  {\sc Normal crossings of centers with strata:} Assume that $X_i\setminus U$ are normal crossings divisors for $i=1,2$. Then, with functoriality understood as in Remark \ref{Rem:ncd-strict}, the divisors $E_i=V_i\setminus U$ are normal crossings and the center $Z_i$ of blowing up in $V_i$ or $V_{i-1}$ has normal crossings with $E_i$ or $E_{i-1}$ as appropriate.

%\ChDan{Furthermore, if  $X_i\setminus U$ are $G$-strict with respect to a given action of a finite group then the same holds for $E_i=V_i\setminus U$.}
%\Dan{Check throughout, especially functoriality! Also check that strictness and $G$-strictness of normal crossings is retained as in \cite[Theorem 3.8 and Remark 3.9]{Borisov-Libgober} \ChDan{(seems to contradict your next comment!)}}\Michael{Check that no use of {\em strict} normal crossings.}

%Let $E_i\subset  V_i$ be the exceptional divisor of $V_i \to X_2$. Then the center $Z_i$ of blowing up in $V_i$ or $V_{i-1}$ has normal crossings with $E_i$ or $E_{i-1}$ as appropriate. If, moreover, $V_i\setminus U$ (respectively, $V_{i-1}\setminus U$) is a normal crossings divisor, then $Z_i$ has normal crossings with  this divisor.}

\item {\sc Conditional factorization in positive and mixed characteristics:} If  functorial embedded resolution of singularities applies in characteristic $p$ (respectively,  over $\ZZ$) for schemes of dimension $\leq d+1$, then there is a functor  $$\Bl({\chara = p,\dim \leq d}) \to \Fact({\chara = p,\dim \leq d})$$ (respectively, a functor $$\Bl({\dim \leq d}) \to \Fact({\dim \leq d}))$$ which is a section of   $\Fact({\chara = p,\dim \leq d})\to \Bl({\chara = p,\dim \leq d})$ (respectively, $\Fact({\dim \leq d})\to \Bl({\dim \leq d})$).

 \end{enumerate}
 \end{theorem}

This generalizes a theorem for {\em varieties}  in characteristic 0, \cite[Theorem 0.3.1 and Remark (3) thereafter]{AKMW}, \cite[Theorem 1.1]{W-ICM}, \cite[Theorem 0.0.1]{W-Seattle}, where the factorization is only shown to be functorial for isomorphisms. The precise statements we need for part (2) are spelled out below as Hypothetical Statements \ref{Hyp:resolution} and \ref{Hyp:principalization}.

\begin{remark}[Preservation of $G$-normality] In \cite[Definition 3.1]{Borisov-Libgober} Borisov and Libgober introduce $G$-normal divisors and in \cite[Theorem 3.8]{Borisov-Libgober} they show that this condition can be preserved in the algorithm of \cite{AKMW}. The same holds true here, using the same argument of   \cite[Theorem 3.8]{Borisov-Libgober}, by performing the sequence of blowings up associated to the barycentric subdivision on the schemes $W_{i\pm}^\res$ obtained in Section \ref{Sec:Tying}. Details are left to the interested reader.
\end{remark}

\subsection{Applications of functoriality}
We need to justify the somewhat heavy functorial treatment. Of course functoriality may be useful if one wants to make sure the factorization is equivariant under group actions and separable field extensions; this has  been of use already in the case of varieties. But it also serves as a tool to transport our factorization result to other geometric spaces.

Blowings up of regular objects is a concept which exists in categories other than schemes, for instance: Artin stacks, qe formal schemes, complex semianalytic germs (see Appendix \ref{App:germs}),  Berkovich $k$-analytic spaces, rigid $k$-analytic spaces.  For brevity we denote the full subcategory of {\em qe noetherian} objects in any of these categories by $\fSp$. Functoriality, as well as the generality of qe schemes, is crucial in proving the following:

\begin{theorem}[Factorization in other categories]\label{Th:factor-Sp} Any blowing up $X_1 \to X_2$ of either noetherian qe regular algebraic stacks, or regular objects of $\fSp$, in characteristic 0 has a weak factorization $X_1 = V_0\das V_1\das\dots \das V_{l-1} \das V_l = X_2.$ The same holds in positive and mixed characteristics (when relevant) if functorial embedded resolution of singularities for qe schemes applies in positive and mixed characteristics.
\end{theorem}

See Theorem \ref{Th:stacks} for the case of stacks and Theorem \ref{Th:main-C} for other categories, where functoriality is also shown, in other words Theorem \ref{Th:main} applies in each of the categories $\fSp$. In addition, the argument deducing Theorem \ref{Th:stacks} from Theorem \ref{Th:main} is a formal one based on functoriality, so the same argument can be used to extend Theorem~\ref{Th:main-C} to stacks in the categories of formal schemes, Berkovich spaces, etc., once an appropriated theory of stacks is constructed, see for instance \cite{Simpson,Noohi, Ulirsch_thesis, Yu, Porta-Yu}.

\subsection{The question of stronger functoriality}\label{Sec:strong-functoriality}
It is natural to replace the category $ \Bl$ by the category $\operatorname{Bl}_r$ with the same objects but where arrows $g:X_2'\to X_2$ as in Definition \ref{Def:Bl} are not required to be surjective, only regular. In a similar way one can replace the category  $\Fact$  by a category $\operatorname{Fact}_r$. As explained in \cite[\S2.3.3]{Temkin} for  resolution of singularities, removing the surjectivity assumption requires imposing an equivalence relation on factorizations, in which two factorizations which differ by a step which is the blowing up of the unit ideal are considered equivalent.  It is conceivable that the analogue of Theorem \ref{Th:main} may hold for $\operatorname{Fact}_r \to \operatorname{Bl}_r$.

\subsection{Factorization of birational and bimeromorphic maps} Our results for projective morphism imply results for birational and bimeromorphic maps. We start with the case of schemes. By a proper birational map $f\:X_1 \das X_2$ of reduced schemes we mean an isomorphism $f_0\:U_1\to U_2$ of dense open subschemes such that the closure $Y\subset X_1\times X_2$ of the graph of $f_0$ is proper over each $X_i$. Assume that $X_i$ are regular. The factorization problem for the birational map $f$ reduces to factorization of the proper morphisms $Y^\res\to X_i$, where $Y^\res$ is a resolution of $Y$. Assume, now, that $f\:X_1\to X_2$ is a proper birational morphism. By a blow up version of Chow's lemma (e.g., it follows from the flattening of Raynaud-Gruson) there exists a blowing up $Y=Bl_I(X_1)\to X_2$ that factors through $X_1$. Then $Y=Bl_{f^{-1}I}(X_1)$ and hence the resolution $Y^\res$, which is the blowing up of $Y$, is also a blowing up of both $X_i$. Thus, factorization of $f$ reduces to the factorization for blowings up, which was dealt with in Theorem~\ref{Th:main}.

Now, assume that $\fSp$ is any geometric category. The definition of a  proper bimeromorphic map $f\:X_1\to X_2$ is similar to the definition of a proper birational map with two addenda: in the case of stacks we require that the morphisms $Y\to X_i$ are representable, and in the case of analytic spaces or formal schemes we require that $U$ is open in $Y$ (in particular, $Y\to X_i$ are bimeromorphic). Then the general factorization problem immediately reduces to the case when $f$ is a proper morphism. Furthermore, if objects of $\fSp$ are compact and if Chow's lemma holds in $\fSp$ then the problem reduces further to the case when $f$ is a blowing up. For complex analytic spaces, Chow's lemma was proved by Hironaka in \cite[Corollary~2]{Hirflat}. It extends immediately to the complex analytic germs we consider in this paper, and these are indeed compact. Most probably, it also holds in all other categories $\fSp$ we mentioned, but this does not seem to be worked out so far.

%Given a proper birational map $X_1 \das X_2$ of nonsingular qe schemes, a result of Hironaka provides morphisms $X_1' \to X_1$ and $X_2' \to X_2$, each given as a sequence of blowings up with nonsingular centers, such that $X_1'\das X_2'$ is a projective morphism, see \cite[Section 1.3]{AKMW}. Thus Theorem \ref{Th:factor-simple} generalizes to the context of proper birational maps. Similarly, in any geometric category $\fSp$ where we have resolution of indeterminacies of bimeromorphic maps, Theorem \ref{Th:factor-simple} generalizes to the appropriate framework of proper bimeromorphic maps.

\section{Qe schemes and functoriality}

\subsection{Projective morphisms and functorial constructions}\label{Sec:functorial-projective}
In our method, it will be important to describe certain morphisms we will obtain as blowing up of a concrete ideal or an explicitly described projective morphism, since further constructions will depend on this data. Moreover, this should be done functorially with respect to surjective regular morphisms. In the current section we develop a few basic functorial constructions of this type.

There are few ways to describe a projective morphism: using $\Proj$, using ample sheaves, or using projective fibrations, but each approach involves choices. Neither description is ``more natural" than the others, and we will have to switch between them. Similarly to \cite[II]{ega} we choose the language of projective fibrations to be the basic one and we will show how other descriptions are canonically reduced to projective fibrations.

\subsubsection{Projective fibrations}
Let $X$ be a scheme. For a coherent $\cO_X$-module $E$ consider the projective fibration $\PP(E)=\PP_X(E) := \Proj_X\Sym^\bullet(E)$ associated with $E$. It has a canonical twisting sheaf $\cO_{\PP(E)}(1)$, and $E\to \pi_* \cO(1)$ is an isomorphism. This construction is functorial for all morphisms: if $\phi\:X'\to X$ is any morphism and $E' = \phi^*E$ then $\PP_{X'}(E') = X'\times_X \PP_X(E)$, and $\cO_{\PP(E')}(1)$ is the pullback of $\cO_{\PP(E)}(1)$.

\subsubsection{Projective morphisms}\label{Sec:projective-morphisms}
By the usual definition \cite[II, 5.5.2]{ega}, a morphism $f\:Y\to X$ is projective if it factors through a closed immersion $i\:Y\into\PP_X(E)$ for a coherent $\cO_X$-module $E$. In this paper, we will use the convention that by saying ``$f$ is projective" we fix $E$ and $i$. In particular, $Y$ acquires a canonical relatively very ample sheaf $\cO_Y(1)=\cO_{\PP(E)}(1)|_Y$. The base change or {\em pullback} $f'\:Y'=Y\times_XX'\to X'$ of $f$ with respect to a morphism $\phi\:X'\to X$ is projective via the embedding $Y'\into\PP_{X'}(E')$, where $E'=\phi^*E$. We will use the notation $f'=\phi^*(f)$. Also, we say that $f$ is {\em projectively the} identity over an open $U$ of $X$ if $E|_U=\cO_U$ and $Y|_U=U$.

\subsubsection{Relation to $\Proj$}
For a projective morphism $f\:Y\to X$ we also obtain a canonical description of $Y$ as a $\Proj$. Namely, if $I_Y\subseteq\cO_{\PP(E)}$ denotes the ideal defining $Y$ then $Y=\Proj_XA$, where $A^\bullet=\Sym^\bullet(E)/I_{Y}$ is a quasi-coherent $\cO_X$-algebra with coherent graded components, generated over $A^0 = \cO_X$ by its degree-1 component $A^1$. Again this structure is functorial for all morphisms:  if $\phi\:X'\to X$ is any morphism and $A' = \phi^*A$ then $\Proj_{X'}A' = X'\times_X \Proj_XA$.

Conversely, if a graded $\cO_X$-algebra $A^\bullet$ has coherent components and is generated over $A^0=\cO_X$ by $A^1$ then $\Sym^\bullet(A^1)\onto A^\bullet$ and we obtain a closed immersion $i\:\Proj_XA\into\PP_X(A^1)$. Thus, $Y=\Proj_XA$ is projective over $X$, and the associated graded quasi-coherent algebra is $A$ itself. This construction is also functorial for all morphisms.

\begin{remark}
We note that the construction of a projective morphism from $\Proj$ is right inverse to the construction of $\Proj$ from a projective morphism, but they are not inverse: going from a projective morphisms to $\Proj$ and back to a projective morphism one usually changes the projective fibration.
\end{remark}

\begin{remark}
In this paper we use superscripts to denote degrees of homogeneous components of a graded object, as in $A^i \subset A^\bullet$. When considering weights of a given $\GG_m$-action we will use subscripts. We hope this will not cause confusion.
\end{remark}

\subsubsection{General $\Proj$}
Consider now a general quasi-coherent graded $\cO_X$-algebra with coherent graded components, which is only assumed to be generated over $A_0 = \cO_X$ in finitely many degrees. Writing $A^{M\bullet} = \oplus_j A^{Mj}$ for a positive integer $M$, we have a canonical isomorphism $Y= \Proj_XA^\bullet \simeq  \Proj_XA^{M\bullet}$. For a suitable $M$ the algebra $A^{M\bullet}$ is generated in degree 1 by $A^M$. If we take the minimal $M_0$ such that $A^{M\bullet}$  is generated in degree 1, then $L$ is not functorial for all morphisms. Rather it is functorial for all {\em flat surjective} morphisms $X'\to X$: if $A^{M\bullet}$ is generated in degree 1 then $(A')^{M\bullet}$ is generated in degree 1, and the opposite is true whenever $X'\to X$ is flat surjective; this follows since surjectivity of $((A')^1)^{\otimes n}\to (A')^n$ implies surjectivity of  $(A^1)^{\otimes n}\to A^n$ by flat decent. Combining this construction with the previous one we obtain an interpretation of $Y\to X$ as a projective morphism, and this construction is functorial for all flat surjective morphisms.

\begin{remark}
This construction applies to the following situation: assume $f\:Y\to X$ is a proper morphism of noetherian schemes and $L$ is an $f$-ample sheaf. Then $A^\bullet = \cO_X \oplus \bigoplus_{k=1}^\infty f_*(L^k)$ is generated in finitely many degrees and $Y=\Proj_XA$. Therefore, $L$ gives rise to an interpretation of $f$ as a projective morphism functorially for all surjective flat morphisms.
\end{remark}

\subsubsection{Blowings up}\label{Sec:blowup-projective}
An important variant is that of blowings up. Consider  a coherent ideal sheaf $I$ on $X$. The Rees algebra $R_X(I)=\oplus_{k=0}^\infty I^k$ is generated in degree 1, and we define $Bl_I(X) = \Proj_XR_X(I)$. In particular, $Bl_I(X)$ is projective over $X$ with the closed immersion $Bl_I(X)\into\PP_X(I)$, and if $I$ is the unit ideal on an open $U$ of $X$ then $Bl_I(X)\to X$ is projectively the identity on $U$. If $\phi\:X'\to X$ is a morphism, then $I^k\cO_{X'} = (I\cO_{X'})^k = (I')^k$ and $\phi^* (I^k) \to  I^k\cO_{X'}$ is surjective, giving a canonical morphism $\phi'\:Bl_{I'}(X')\to Bl_I(X)$ over $\phi$. Clearly $(\phi')^*L = L'$. So a blowing up is functorially projective. If moreover $X' \to X$ is flat, then $Bl_{I'}(X') = X'\times_X Bl_XI$.

We will need an opposite construction, using a variant of \cite[Theorem II.7.17]{Hartshorne} for regular schemes. Assume $X$ is regular and $f\:Y \to X$ is a proper birational morphism with an ample sheaf $L$ (e.g., if $Y\to X$ is projective we can take $L=\cO_Y(1)$). Then after replacing $L$ by a positive power which is functorial for flat surjective morphisms, we have that  $Y = \Proj_X A^\bullet$, where $A^\bullet$ is generated over $A_0 = \cO_X$ by its degree-1 component, and $A^k = f_*L^k$.

Locally on $X$, write  $L^k$ as a fractional ideal on $Y$, giving it as a fractional ideal $F_{L,k}$ on $X$ since $Y \to X $ is birational. Since $A^\bullet$ is generated in degree 1, we have that $F_{L,k} = F_{L,1}^k$ (see  \cite[Theorem II.7.17 Step 5]{Hartshorne}). Since $X$ is factorial, there is a unique expression $F_{L,1} = M I$, where $M$ is an invertible fractional ideal and $I$ is an {\em ideal sheaf} without invertible factors. Explicitly, $F_{L,1}^*$ is invertible, so we can write $I = F_{L,1}^* F_{L,1}$ and $M=F_{L,1}^{**}$. It follows that $F_{L,k} =M^k I^k$. Note that while the construction is local on $X$ and depends on an embedding of $L$ in the fraction field, the ideal sheaf $I$ glues canonically. Locally on $X$ we have a canonical isomorphism $Y \simeq Bl_I(X)$, which evidently glues canonically. We have obtained that a projective birational morphism $f\:Y \to X$ with $X$ regular is a blowing up, functorially for flat surjective morphisms $X' \to X$ of regular schemes. In addition, if $f$ is projectively the identity on $U\subseteq X$ then $I$ is the unit ideal on $U$.

For future reference we record the following well known result that follows from the universal property of blowings up.

\begin{lemma}\label{factorblowlem}
If $X$ is an integral scheme and a blowing up $Y=Bl_I(X)\to X$ factors through a proper birational morphism $Z\to X$ then $Y=Bl_{I\cO_Z}(Z)$.
\end{lemma}

\subsubsection{Sequences of projective morphisms}\label{Sec:sequence-projective}
Now assume $Z\stackrel{g}\to Y \stackrel{f}\to X$ is a sequence of projective morphisms of noetherian schemes, say $Z\into\PP_Y(F)$ and $Y\into\PP_X(E)$ for a coherent $\cO_Y$-module $F$ and a coherent $\cO_X$-module $E$. For a large enough $n$ the map $f^*f_*(F(n))\stackrel{\alpha}\to F(n)$ is surjective, hence $\PP_Y(F)=\PP_Y(F(n))$ embeds into $\PP_X(E\otimes f_*F(n))$ and we obtain a closed immersion $Z\into\PP_X(E\otimes f_*F(n))$. Choosing the minimal $n$ such that $\alpha$ is surjective we obtain a construction that realizes composition of projective morphisms as a projective morphism functorially for flat surjective morphisms $X'\to X$.

If $X$ is regular we can combine this with the previous statements, so if $Y_m\to \cdots\to Y_1 \to X$ is a sequence of birational projective morphisms which are projectively the identity over an open $U\subseteq X$, then $Y_m\to X$ is a blowing up of an ideal sheaf which is the unit ideal on $U$, and this is functorial for flat and surjective morphisms of regular schemes.

\begin{remark}
We will not use this, but blowings up can also be composed in terms of ideals. One can show that if $X$ is normal then the composition of $Y=Bl_I(X)\stackrel{f}\to X$ and $Bl_J(Y)\to Y$ is of the form $Bl_{f_*(f^{-1}(I^n)J)}(X)\to X$ for a large enough $n$.
\end{remark}

\subsection{Qe schemes and resolution of pairs}\label{Sec:qe-resolution}

\subsubsection{Qe schemes}
The class of quasi-excellent schemes was introduced by Grothen\-dieck as the natural class where problems related to resolution of singularities behave well. The name ``quasi-excellent" is perhaps not very elegant (it was not introduced by Grothendieck), and we feel it harmless to refer to them as {\em qe schemes}.

First recall  that regular morphisms\index{regular morphism} are a generalization of smooth morphisms in situations of morphisms which are not necessarily of finite type. Following \cite[$\rm IV_2$, 6.8.1]{ega}  a morphism of schemes $f\:Y\to X$ is said to be  {\em regular} if
\begin{itemize}
\item  the morphism $f$ is flat and
\item all geometric fibers of $f\: Y \to X$ are regular.
\end{itemize}

A locally noetherian scheme  $X$ is a {\em qe scheme}\index{qe scheme} if the following two conditions hold:
\begin{itemize}
\item for any scheme $Y$ of finite type over $X$, the regular locus $Y_\reg$ is open; and
\item for any point $x\in X$, the completion morphism $\Spec \hat\cO_{X,x} \to \Spec \cO_{X,x}$ is regular.
\end{itemize}

It is a known, but nontrivial fact, that a scheme $Y$ of finite type over a qe scheme is also a qe scheme, see, for example, \cite[34.A]{Matsumura}. A ring $A$ is a qe ring if $\Spec A$ is a qe scheme.

\subsubsection{Resolution of pairs}
Consider a pair $(X,Z)$, where $X$ is a reduced qe scheme and $Z$ is a nowhere dense closed subset of $X$. By a {\em resolution} of $(X,Z)$ we mean a birational projective morphism $f\:X'\to X$ such that $X'$ is regular, $Z'=f^{-1}(Z)$ is a simple normal crossings divisor, and $f$ is projectively the identity outside of the union of $Z$ and the singular locus $X_\sing$ of $X$. Since \cite[IV$_2$, 7.9.6]{ega}, it is universally hoped that every qe scheme admits a good resolution of singularities; the same should also hold for pairs, see  Remark \ref{Rem:how-to-resolve-pairs} below. If $X$ is noetherian of characteristic zero then $(X,Z)$ can be resolved by \cite[Theorem~1.1]{Temkin-qe}.

\begin{remark}\label{Rem:how-to-resolve-pairs}
(i) Usually, resolution of pairs is constructed in two steps:
\begin{itemize}
\item[(1)] Resolve $X$ by a projective morphism $f\:X'\to X$. Usually, this is achieved by a sequence of blowings up $X_l\to\dots\to X_0=X$. One can also achieve that the centers are regular, though this requires an additional effort.
\item[(2)] Resolve $Z'=f^{-1}(Z)$ by a further projective morphism $f'\:X''\to X'$. Usually, this is achieved by a sequence of blowings up $X''=X'_n\to\dots\to X'_0=X'$ whose centers are regular and have simple normal crossings with the accumulated exceptional divisor, so that all schemes $X'_i$ remain regular and exceptional divisors $E'_i$ are simple normal crossings. In addition, one achieves a {\em principalization} of $Z'$ as a subscheme, i.e. $Z'\times_{X'}X'_n$ is a divisor supported on   $E'_n$.
\end{itemize}

(ii) The best known results for general noetherian qe schemes beyond characteristic 0 are resolution of qe threefolds, see \cite{CP}, and principalization of surfaces in regular qe schemes, see \cite{CJS}. In particular, a noetherian qe pair $(X,Z)$ can be resolved whenever $\dim(X)\le 3$.
\end{remark}

\subsubsection{Compatibility with morphisms}
By a morphism of pairs $\phi\:(Y,T)\to(X,Z)$ we will always mean a morphism $\phi\:Y\to X$ such that $T=\phi^{-1}(Z)$. We say that resolutions $f_X\:X'\to X$ and $f_Y\:Y'\to Y$ of $(X,Z)$ and $(Y,T)$ are {\em compatible} with $\phi$ if $f_Y=\phi^*(f_X)$.

\begin{remark}
As we mentioned, often resolution of pairs has a natural structure of a composition of blowings up. The definition of compatibility in this case is similar with the only difference that the blowing up sequence of $Y$ is obtained from the pullback of the blowing up sequence of $X$ by removing all blowings up with empty centers. The latter contraction procedure is only needed when $f$ is not surjective.
\end{remark}

\subsubsection{Functorial resolution}
Let $\cC$ be a class of pairs $(X,Z)$, where $X$ is a reduced noetherian qe scheme and $Z$ is a closed subscheme. Throughout this paper, by a {\em functorial resolution} on $\cC$ we mean a rule that assigns to any pair $(X,Z)\in\cC$ a resolution $(X',Z')\to(X,Z)$ in a way compatible with arbitrary {\em surjective regular} morphisms between pairs in $\cC$. In addition, we always make the following assumption on the resolution of {\em normal crossings pairs}, i.e. pairs $(X,Z)$ with regular $X$ and normal crossings $Z$ (not necessarily simple):

\begin{assumption}\label{Ass:snc}
For any normal crossings pair $(X,Z)$ in $\cC$ its resolution $X'\to X$ can be functorially represented as a composition of blowings up whose centers are regular and have normal crossings with the union of the preimage of $Z$ and the accumulated exceptional divisor.
\end{assumption}

\begin{remark}\label{functorrem}
(i) This definition provides the minimal list of properties we will use. As we remarked earlier usually one proves finer desingularization results obtaining, in particular, that $Z\times_XX'$ is a divisor and the resolution is functorial for non-surjective morphisms as well.

(ii) It seems that any reasonable resolution should satisfy the assumption. In fact, most (if not any) algorithms appearing in the literature apply to normal crossings pairs $(X,Z)$ via the following {\em standard algorithm}: first one blows up the maximal multiplicity locus of $Z$, then one blows up the maximal multiplicity locus of the strict transform of $Z$, etc. It is easy to see that the standard algorithm satisfies the assumption.
\end{remark}

\subsubsection{Resolution of singularities of qe schemes: characteristic 0}\label{Sec:resolution}
Functorial resolution of pairs is known in characteristic zero:

\begin{theorem}\label{resolth}
There exists a functorial resolution, satisfying Assumption \ref{Ass:snc}, on the class $\cC_{\chara=0}$ whose elements are pairs $(X,Z)$ with $X$ a reduced noetherian qe scheme over $\QQ$.
\end{theorem}
\begin{proof}
By \cite[Theorem~1.1.7]{Temkin-embedded} there exists a blowing up sequence $$\cF_{\rm princ}(X,Z)\:X'\to\dots\to X$$ whose centers lie over $Z\cup X_\sing$ and such that $X'$ is regular and $Z'=f^{-1}(Z)$ is a simple normal crossings divisor. Moreover, this sequence is functorial in regular morphisms. By \S\ref{Sec:sequence-projective}, the morphism $X'\to X$ is a projective morphism functorially in surjective regular (even flat) morphisms. Finally, a direct (but tedious) inspection shows that the algorithm $\cF_{\rm princ}$ of loc.cit. resolves normal crossings pairs via the standard algorithm.
\end{proof}

\begin{remark}
Functoriality of this resolution implies that one also gets a functorial way to resolve an arbitrary qe pair over $\QQ$ (locally noetherian but not necessarily noetherian) by a morphism $f\:X'\to X$. In general, there is no natural way to provide $f$ with an appropriate structure, neither as a single blowing up nor a sequence of blowings up. However, $f$ can be realized as an infinite composition whose restrictions onto noetherian open subschemes of $X$ are finite, e.g., the case of $Z=\emptyset$ is worked out in \cite[Theorem 5.3.2]{Temkin}.
\end{remark}

\subsubsection{Positive and mixed characteristics hypothesis} In Theorem \ref{Th:main} (3), the precise hypothetical statement we need about resolutions of pairs is the following analogue of Theorem~\ref{resolth}:

\begin{hypothesis}\label{Hyp:resolution} \begin{enumerate}
\item {\sc Functorial resolution:} The class $\cC_{\chara=p,\dim\leq d+1}$ (resp. $\cC_{\dim\leq d+1}$) of pairs $(X,Z)$, where $X$ is reduced noetherian qe $\FF_p$-scheme (respectively, $\ZZ$-scheme) of dimension $\leq d+1$, admits a functorial resolution $f_{(X,Z)}\:X'\to X$ satisfying Assumption~\ref{Ass:snc}.
\item {\sc $\GGm$-equivariance:} Moreover, the resolution is compatible with any $\GGm$-action on $(X,Z)$ in the sense that $a^*(f_{(X,Z)}) = p_X^*(f_{(X,Z)})$, where $a:\GGm\times X\to X$ is the action morphism and $p_X:G\times X \to X$ is the projection.
%\Dan{\color{green} seems to me that by taking slices of $G\times X$ at every point  one should be able to deduce (2) from (1). The point is that the slices one takes are tor-independent from the ideals, so if they coincide on the slice I hope they must coincide on $G\times X$}
\end{enumerate}
{\rm In mixed characteristics we will also need:}
\begin{enumerate}
\item[(3)] {\sc Functoriality of toroidal charts:} assume that $X$ is a toroidal scheme (\cite[\S2.3.4]{AT1}) of dimension at most $d+1$ and $j\:X\to Y=\Spec \ZZ[M]$ is a toroidal chart (\cite[\S2.3.17]{AT1}), $T$ is a toric subscheme of $Y$ and $Z=X\times_YT$, then $j^*(f_{(Y,T)}) = f_{(X,Z)}$.
%\Dan{in fact, we only need to resolve: (1) $\GG_m$-equivariantly the blowing up of $I + I_{\{0\}}$ on $B_\cO$ of dimension $d+1$, where the blowing up $X_1$ of $I$ is regular, and (2) functorially and equivariantly the quotients $W_i$, which are locally toroidal of dimension $d$}
\end{enumerate}
\end{hypothesis}

We note that the equivariance statement (2) in dimension $d+1$ follows from statement (1) in dimension $d+2$, but here we wish to only make assumptions up to dimension $d+1$. It is conceivable that a version of (2) sufficient for our needs follows from (1) by taking slices, but we will not pursue this question.

Let us say that a pair $(X,Z)$ is {\em locally monoidal} if locally $X$ admits a logarithmic structure making it to a logarithmically regular scheme so that the ideal of $Z$ is monoidal. It is expected that there should exist a canonical resolution of such pairs of combinatorial nature, which is, in particular, independent of the characteristics. Our Statement (3) asserts such independence in mixed characteristics; in pure characteristics it is a consequence of equivariance. It is analogous to Hypothetical Statement \ref{Hyp:principalization}(3) below. Similarly to Hypothetical Statement \ref{Hyp:principalization}, proving Statements (1)--(3) for locally monoidal pairs is expected to be easier than the general case. For example, it is proved in \cite[Theorem~3.4.9]{Illusie-Temkin} for logarithmically regular schemes (with a single logarithmic structure), but the known functoriality \cite[Theorem~3.4.15]{Illusie-Temkin} is not enough to extend it to locally monoidal schemes. In addition, very recently Buonerba resolved certain locally monoidal varieties in \cite{Buonerba}.

\subsection{Principalization of ideal sheaves}\label{Sec:principalization} In addition to resolution of pairs, we will need a version of functorial principalization of coherent ideal sheaves on a qe regular scheme $X$ with a simple normal crossings divisor $D$, that will often be called the {\em boundary}. In fact, we will only need a particular case of locally monoidal ideals as introduced below.

\subsubsection{Permissible sequences}
A blowing up sequence $X_n\to\dots\to X_0=X$ will be called {\em permissible} if its centers $V_i\subset X_i$ are regular and have simple normal crossings with $D_i\subset X_i$, which is defined to be the union of the preimage of $D$ and the accumulated exceptional divisor. Note that in such case each $X_i$ is regular and each $D_i$ is a boundary.

\subsubsection{Principalization}
We consider the category of triples $(X,D,I)$ where $(X,D)$ is a noetherian regular qe scheme with a boundary, $I$ is a coherent ideal sheaf, and arrows are regular morphisms $f\:X' \to X$ such that $f^{-1} I = I'$ and $f^{-1} D = D'$. A {\em principalization} of $I$ is a permissible sequence of blowings up $\phi_{(X,D,I)}\:X_n\to\dots\to X_0=X$ such that:
\begin{itemize}
\item[(1)] Each center $V_i$ lies in the union of $D_i$ with the locus where $I$ is not the unit ideal.
\item[(2)] $I_n=\phi_X^{-1}I$ is a divisorial ideal supported on $D_n$. In particular, $V(I_n)$ is a divisor with a simple normal crossings reduction.
\end{itemize}
Principalizations form a category again, and  functorial principalization provides a functor from triples $(X,I,D)$ to principalizations $\phi_X\:X'\to X$. As we do not require the morphism $f$ to be surjective, we have to use the  equivalence relation mentioned in Section \ref{Sec:strong-functoriality}. However, we will only apply the result in the context of  surjective morphisms, so this equivalence will not figure in any of our applications.

\subsubsection{Known results}
Functorial principalization of ideal sheaves for {\em varieties} over a field of characteristic 0 is known, e.g., see \cite[Sections 11,13]{Bierstone-Milman} or \cite[Theorem 3.26]{Kollar}. The second author is in the process of writing a general functorial principalization of ideal sheaves on noetherian regular qe schemes over $\QQ$ with the methods of \cite{Temkin-embedded}; we will manage not to use this result. For general qe schemes, the best known result is principalization on threefolds.

\begin{remark}\label{orderrem}
(i) Classically, one only blows up centers over the locus where $I$ is not trivial. On the other hand, usually one works with ordered boundaries $D=\cup_{i=0}^nD_i$, where $D_i$ are smooth components. Ordering the boundary restricts functoriality and, in fact, it is not critical. For example, the boundaries in \cite{CJS} are not ordered.

(ii) Since we allow blowings up that modify the whole $D$, we can freely use the classical results to resolve $(X,D,I)$: first apply the standard principalization $f\:X_n\to\dots\to X$ to $(X,D)$, then $D_n$ is a simple normal crossings divisor ordered by the history of blowings up, and we can apply a classical algorithm to $(X_n,D_n,f^{-1}I)$.
\end{remark}

\subsubsection{Locally monoidal ideals}
A triple $(X,D,I)$ with $X$ regular, $D$ a boundary and $I$ an ideal sheaf on $X$ is said to be {\em locally monoidal} if there is an open covering $\coprod U_\alpha \to X$, logarithmically regular structures $(U_\alpha, M_\alpha)$ in the sense of \cite{Kato-toric} and \cite[\S 2.3.1]{AT1} such that $D$ is part of the toroidal divisor, and monoid ideals $I_\alpha \subset M_\alpha$ such that $I_{U_\alpha}$ is generated by the image of $I_\alpha$ under $M_\alpha\to \cO_{U_\alpha}$.

\begin{hypothesis}\label{Hyp:principalization}
\begin{enumerate}
\item Each locally monoidal $\FF_p$-triple (respectively, $\ZZ$-triple) $(X,D,I)$ of dimension $\leq d$ admits a principalization $$\phi_{(X,D,I)}\:\tilde X\to\dots\to X$$ in a manner functorial for regular morphisms $X'\to X$.
\item Moreover, if $a:\GG_a^d\times X \to X$ is an action such that $I$ and $D$ are equivariant: $a^{-1}I = p_X^{-1} I$ and $a^{-1}D = p_X^{-1} D$, where $p_X:\GG_a^d\times X \to X$ is the projection, then $\tilde X \to X$ is $\GG_a^d$-equivariant as well.
\end{enumerate}
{\rm Again in mixed characteristics we also need:}
 \begin{enumerate}
\item[(3)] {\sc Functoriality of toroidal charts:} assume that $(X,D,I)$ is locally monoidal of dimension $\le d$ and $j\:(X,D)\to(\bfA_M,\bfA_{M^\gp})$ is a toroidal chart such that $I=j^{-1}I_0$ for a toric ideal $I_0$ on $\bfA_M$. Then the sequence $\phi_{(X,D,I)}$ is the pullback of $\phi_{(\bfA_M,\bfA_{M^\gp},I_0)}$.
\end{enumerate}
\end{hypothesis}

\begin{remark}
(i) In fact, the hypothesis asserts that toric ideals on schemes $\Spec\ZZ[M]$ can be principalized so canonically that given a locally monoidal triple $(X,D,I)$ any toroidal chart induces the same principalization of $I$.

(ii) We remark that the results of \cite[Section 3.1.14]{Illusie-Temkin} suggest that this statement may be within reach: in that paper the local non-functorial problem is solved, and the problem reduces to making the process functorial even if one changes the logarithmic structure $M_\alpha$ on $U_\alpha$.
\end{remark}

\subsubsection{The characteristic zero case}
To make our results unconditional in characteristic zero we should prove that parts (1) and (2) of \ref{Hyp:principalization} hold for schemes over $\QQ$. In fact, we will even deal with a larger class of triples using the case of varieties and methods of \cite[Theorem 2.4.1, p. 95]{Illusie-Temkin}.

A triple $(X,D,I)$ is said to be {\em $\QQ$-absolute} if there exists an open covering $\coprod U_\alpha \to X$, regular $\QQ$-varieties $Z_\alpha$, regular morphisms $f_\alpha\:U_\alpha \to Z_\alpha$, ideal sheaves $I_\alpha$ on $Z_\alpha$ and divisors $D_\alpha \subset Z_\alpha$ such that $f_\alpha^{-1} I_\alpha = I|_{U_\alpha}$ and $f_\alpha^{-1} D_\alpha = D|_{U_\alpha}$. The collection of $\QQ$-absolute triples forms a full subcategory of the category of triples. Functorial  principalization of $\QQ$-absolute triples $(X,D,I)$ is a functor from this subcategory to principalizations of the corresponding ideals.

The statement we need is the following:

\begin{proposition}\label{Prop:absolute-principalization}
There exists a functorial principalization $\phi_X\:\tilde X\to X$ of $\QQ$-absolute triples $(X,D,I)$.
\end{proposition}
\begin{proof}
We may replace $\coprod U_\alpha$ by a finite covering, since $X$ is noetherian. We write $U_{\alpha\beta} = U_\alpha \times_{X} U_\beta$. Now, we will use the ideas from the proof of \cite[Theorem~2.4.3]{Illusie-Temkin}.

First we construct a principalization. For this it suffices to construct a principalization of $\coprod(U_\alpha, D|_{U_\alpha}, I|_{U_\alpha})$ whose two pullbacks to the fiber product $W:=\coprod U_{\alpha\beta}$ coincide. The triple $(Z,D_Z,I_Z):=\coprod(Z_\alpha,D_\alpha,I_\alpha)$ has a principalization compatible with $D_\alpha$ coming from the principalization functor for $\QQ$-varieties. This pulls back to a principalization of  $\coprod(U_\alpha, D|_{U_\alpha}, I|_{U_\alpha})$ and we need to show that the two pullbacks to $W$ coincide. We have two regular morphisms $f,g\:W\to Z$. By Popescu's theorem (see \cite{Popescu} or \cite{Spivakovsky}), $f$ is the limit of smooth morphisms $f_\gamma\:W_\gamma\to Z$. By \cite[IV$_3$, Proposition~8.13.1]{ega}, $g$ factors through a morphism $g_\gamma\:W_\gamma\to Z$ for a large enough $\gamma$ and then \cite[Proposition~2.4.3]{Illusie-Temkin} implies that replacing $W_\gamma$ by a neighborhood of the image of $W$ we can achieve that $g_\gamma$ is also smooth. Since the two pullbacks of $I_Z$ and $D_Z$ to $W$ coincide, there is some $\gamma$ such that  the two pullbacks of $I_Z$ and $D_Z$  to $W_\gamma$ coincide. It follows by functoriality of principalization for varieties that the two principalizations on $W_\gamma$ coincide, and therefore they coincide on $W$, as required.

We now demonstrate that this principalization is functorial. Consider a regular surjective morphism $f\:(X_1, D_1,I_1) \to (X_2, D_2,I_2)$ with coverings $\coprod U_{1\alpha}$ and $\coprod U_{2\beta}$ and $\QQ$ varieties $Z_{1\alpha}$ and $Z_{2\alpha}$. Then composing $U_{2\beta}\to Z_{2\beta}$ with $f$ we get another covering $\coprod f^{-1}U_{2\beta}$ with regular maps to $Z_{2\beta}$, so it is enough to show that the resulting principalizations on $X_1$ coincide. We now write $W = \coprod U_{1\alpha}\times_{X_1}f^{-1}U_{2\beta}$, which maps to $Z_1= \coprod Z_{1\alpha}$ and $Z_2= \coprod Z_{2\beta}$. By the same argument as earlier we have that $W\to Z_1\times Z_2$  is the limit of a family $W_\gamma \to Z_1\times Z_2$, where the two maps $W_\gamma\to Z_i$ are smooth. As above we conclude that the ideals  and divisors coincide on some $W_\gamma$ and  the two principalizations coincide on $W$ and therefore on $X_1$.
\end{proof}

\section{Functorial toroidal factorization}
\subsection{Statement} We follow the treatment of toroidal schemes in \cite[Section 2.3]{AT1}, in particular they carry logarithmic structures in the Zariski topology. A toroidal ideal $I$ on a toroidal scheme $X$ with logarithmic structure $M$ is the ideal generated by the image of a monomial ideal in $M$ through $M\to \cO_X$.  We define a category $\TorBl$ of toroidal blowings up, similar to $\Bl$:
\begin{enumerate}
\item  An object is a birational transformation $X_1 \to X_2$ where $X_1, X_2$ are toroidal and regular, and $X_1 \to X_2$ is given as the normalized blowing up of a toroidal ideal $I\subset \cO_{X_2}$.
\item An arrow from $X'_1 \to X'_2$ to $X_1 \to X_2$  consists of a regular surjective morphism $g:X_2' \to X_2$, such that $U_{X_2} = g^{-1}U_{X_2}$ and $I' = I\cO_{X_2'}$.
\end{enumerate}
We similarly define a {\em toroidal weak factorization} $X_1 = V_0\das V_1\das\dots \das V_{l-1} \das V_l = X_2$ of a toroidal blowing up $X_1 \to X_2$, where the schemes $V_i$, ideals $J_i$ and centers $Z_i$ are toroidal. These form {\em the regular surjective  category $\TorFact$ of toroidal weak factorizations} in a manner similar to the above.

\begin{proposition}\label{Prop:toroidal-factorization}
Let $X_1 \to X_2$ be a toroidal morphism of toroidal schemes obtained by normalized blowing up a toroidal ideal. Then there is a toroidal weak factorization  $X_1 = V_0\das V_1\das\dots \das V_{l-1} \das V_l = X_2$ in a functorial manner:  there is a section $\TorBl \to \TorFact$ of the forgetful functor $\TorFact \to\Bl$.
\end{proposition}

\begin{remark}
Jaros\l aw W\l odarczyk informed us that  one can prove a stronger result:  a factorization procedure which is  functorial for all regular strict morphisms $g:X_2' \to X_2$, not required to be surjective. His proposed argument involves subtle modifications at the heart of the algorithm in  \cite[Sections 4 and 5]{W-Seattle}. The proof we provide at the end of this section shows that {\em any} procedure for toric factorization gives rise to a functorial procedure.
\end{remark}

\subsection{Cone complexes} Before proving Proposition \ref{Prop:toroidal-factorization} we need to discuss a generalization of the  polyhedral cone complexes with integral structure of \cite{KKMS} which was introduced in \cite[2.5]{ACP} to accommodate any toroidal embedding in the sense of \cite{KKMS}, allowing for self intersections and monodromy. In this paper we only assign  polyhedral cone complexes to Zariski toroidal schemes,  without self intersections or monodromy, but the generalized  polyhedral cone complexes are used as a combinatorial tool to achieve functoriality.

Fix a toroidal scheme $X$. Recall that the polyhedral complex of  \cite{KKMS} or the equivalent Kato fan of \cite{Kato-toric} assigns  a polyhedral cone $\sigma_{Z}$ with integral structure to each toroidal stratum $Z \subset X$; each inclusion $Z' \hookrightarrow \overline Z\subset X$ gives rise to a linear map $\nu:\sigma_Z \to \sigma_{Z'}$, which identifies $\sigma_Z$ as a face of $\sigma_{Z'}$ in such a way that the integral structure on $\sigma_{Z}$ is the restriction of the integral structure of $\sigma_{Z'}$: this is called a {\em face map}. We define $\Sigma(X) = \varinjlim(\{\sigma_Z\},\{\nu\})$ - it is similar to the fan of a toric variety, but is not embedded in a space $N_\RR$ and the intersection of two cones may be the union of faces rather than just one face.

A map of polyhedral cone complexes $\varinjlim(\{\sigma'_i\},\{\nu'_k\})\to \varinjlim(\{\sigma_j\},\{\nu_l\})$ is defined to be a collection of cone maps $\sigma'_i \to \sigma_{j(i)}$ compatible with the face maps $\nu'_k$ and $\nu_k$. A toroidal map $X'\to X$ gives rise to a map of cone complexes; here are a few well known relationships:

\begin{enumerate}
\item A proper birational toroidal morphism gives rise to a subdivision, and there is an equivalence of categories between proper toroidal birational morphisms and subdivisions. Blowings up of ideals correspond to subdivisions determined by piecewise linear continuous integral functions which are convex on each cone; following \cite{KKMS} we call these  {\em projective subdivisions} (in the combinatorial literature they are {\em coherent subdivisions}).
\item A regular morphism $g:X_2' \to X_2$ such that $U_{X_2} = g^{-1}U_{X_2}$ gives rise to a map of complexes $\Sigma(g):\Sigma(X') \to \Sigma(X)$  where all the maps $\sigma'_i \to \sigma_{j(i)}$ are face maps - this is called a face map of complexes.
\item If the map $g:X_2' \to X_2$ is also surjective then $\Sigma(g)$ is surjective.
\item  The scheme $X$ is regular if and only if all the cones $\sigma_i\subset \Sigma(X)$ are  nonsingular in the usual toric sense.
\item  If $X$ is regular then the closure of a stratum is always regular (this would fail if we allowed self intersections); we call such subschemes {\em toroidal centers}.
\item  The blowing up   $X'\to X$  of an irreducible toroidal center  $\overline Z$ on a {\em regular} $X$ corresponds to the star subdivision $\Sigma' \to \Sigma(X)$ at the barycenter of $\sigma_Z$. The blowing up   $X'\to X$  of any regular toroidal subscheme $W$ corresponds to the simultaneous star subdivision $\Sigma' \to \Sigma(X)$ at the barycenters of all the cones  corresponding to the connected components of $W$.
\end{enumerate}

Thus proposition  \ref{Prop:toroidal-factorization} would follow if the projective subdivision $\Sigma(X_1) \to \Sigma(X_2)$ can be factored as a composition of such simultaneous star subdivisions and their inverses,  in such a way that the intermediate steps are projective subdivisions of $\Sigma(X_2)$, in a functorial manner with respect to surjective face maps. This will be our Lemma \ref{Lem:generalized-factorization-functorial} below.

Morelli's  $\pi$-desingularization lemma of fan cobordisms \cite[Lemma 10.4.3]{W-Factor} gives a non-functorial result in the case of fans; this was generalized in \cite{Abramovich-Matsuki-Rashid} to polyhedral cone complexes. In \cite{AKMW} it is made functorial under {\em automorphisms,} which is not sufficient for our purposes here.

Consider the category whose objects are  projective subdivisions  $\Sigma_1 \to \Sigma_2$ of nonsingular cone complexes given by a fixed piecewise linear continuous integral function $f\:\Sigma_2\to\RR$ convex on each cone and arrows $(\Sigma'_2, f') \to (\Sigma_2,f)$ induced by surjective face maps  $h:\Sigma_2' \to \Sigma_2$  with $f' = f\circ h$.
Functoriality would be easily achieved if the connected component of any object $\Sigma_1 \to \Sigma_2$ in this  category had a final object, as we show below in Lemma \ref{Lem:generalized-factorization-functorial}. Indeed, this would mean that applying Morelli's lemma to the final object would induce a factorization for the whole component, giving the result. Unfortunately final objects usually do not exist in the category of cone complexes. Our next goal is to enlarge this category so that final objects do exist, see Lemma \ref{Lem:final} below.

\subsection{Generalized cone complexes and existence of final objects} A generalized cone complex is given by any finite diagram $(\{\sigma_j\},\{\nu_l\})$ of cones and face maps. We allow for more than one face map $\sigma_j \to \sigma_l$, including non-trivial self-face maps $\sigma_j\to \sigma_j$. We think of a generalized cone complex $\Sigma$ as a structure imposed on the topological space $\underline{\Sigma} =\varinjlim(\{\sigma_j\},\{\nu_l\})$. Thus an arrow of generalized cone complexes $(\{\sigma'_i\},\{\nu'_k\})\to (\{\sigma_j\},\{\nu_l\})$ is given by compatible cone maps as above; an arrow is a face map if it is given by compatible face maps; and an arrow is declared to be an isomorphism if it is a face map inducing a bijection of sets  $\varinjlim(\{\sigma'_i\},\{\nu'_k\})\to \varinjlim (\{\sigma_j\},\{\nu_l\})$. See \cite[\S 2.6]{ACP}.

Cone complexes are a full subcategory of generalized cone complexes. They are distinguished by the property that, for any cones $\tau,\sigma$ of  $\Sigma$ a face map $\nu:\tau\to \sigma$ in $\Sigma$ is unique if it exists.
Thus proposition  \ref{Prop:toroidal-factorization} would again follow if any projective subdivision $\Sigma_1 \to \Sigma_2$ of {\em generalized} nonsingular cone complexes can be factored as a composition of  simultaneous star subdivisions and their inverses, in a functorial manner with respect to surjective cone maps. The advantage of working with generalized cone complexes is the following:

\begin{lemma}\label{Lem:final}
The connected component of the projective subdivision $\Sigma_1 \to \Sigma_2$ of generalized cone complexes in the category induced by surjective face maps $\Sigma'_2 \to \Sigma_2$ has a final object.
\end{lemma}
\begin{proof}
The projective subdivision  $\Sigma_1 \to \Sigma_2$ is induced by an implicit piecewise linear convex integral function $f:\Sigma_2 \to \RR$. Write $\Sigma_2 =  (\{\sigma_j\},\{\nu_l\})$. Then $\nu_l: \sigma_i \to \sigma_j$  has the property that $f_{\sigma_i} = f_{\sigma_j} \circ \nu_l$. Let $\{\mu_k\}$ be the collection of all face maps $\mu_k: \sigma_m \to \sigma_n$ with the property that $f_{\sigma_m} = f_{\sigma_n} \circ \mu_k$. Then  $\Delta :=  (\{\sigma_j\},\{\mu_k\})$ is a generalized cone complex, the maps $f_{\sigma_j}$ glue to give a  piecewise linear integral function $\tilde f: \Delta\to \RR$,  and since  $\{\nu_l\}\subset \{\mu_k\}$ we have a  map of diagrams $g:\Sigma_2 \to \Delta$ such that $f= \tilde f\circ g$.

It is convenient to have another presentation of $\Delta$. Choose one representative $\bar\sigma$ from each isomorphism class of cones in $\Delta$. Given two such representatives $\bar\tau$ and $\bar\sigma$, consider all maps $\bar\nu_l: \bar\tau \to \bar\sigma$ in $\Delta$. Clearly $\bar\Delta = (\{\bar\sigma\},\{\bar\nu_l\})$ maps as a subdiagram to $\Delta$, and the map is  an isomorphism since it is clearly a bijection on set theoretic limits.

We claim that $(\Delta,\tilde f)$ is a final object in the component of $(\Sigma_2,f)$ in the category of generalized cone complexes with piecewise linear integral function. For this it suffices to show that if $(\Sigma_2', f')$ is an object and $h:\Sigma_2 \to \Sigma_2'$ is a surjective face map such that $f'\circ h = f$ then $g= g''\circ h$ where $g'':   \Sigma_2'\to \Delta$ is a morphism so that $f' = \tilde f \circ g''$.

First, if we apply the construction of $\Delta$ to $\Sigma_2'$ we get a map $g':\Sigma_2'\to \Delta'$ which sits in a commutative diagram
$$\xymatrix{
\Sigma_2 \ar^g[rr]\ar_h[dd]\ar_f[rrrd] && \Delta\ar_{\tilde h}[dd]\ar^{\tilde f}[rd]\\ &&&\RR\\  \Sigma_2' \ar_{g'}[rr]\ar^{f'}[rrru] && \Delta'\ar_{\tilde f'}[ru]
}$$

On the other hand $\bar\Delta\simeq \Delta$ and $\bar\Delta' \simeq \Delta'$, and the  map $\bar\Delta \to \bar\Delta'$ induced by $\tilde h$ is an isomorphism of diagrams: since $h$ is a surjective face map, any cone in $\Sigma_2'$ is isomorphic to a cone of $\Sigma_1$ via an isomorphism compatible with $f$ and vice versa. So $\tilde h$ gives a bijection between the isomorphism classes of cones, and the maps $\bar\nu$ between cones are determined by the compatibility of the function $\tilde f =\tilde f'$ on them. So $\Delta\to \Delta'$ is an isomorphism, giving the requisite map of generalized complexes $g'' = \tilde h^{-1} \circ g'$.
\end{proof}

\subsection{Barycentric subdivisions and factorization for generalized cone complexes}

We proceed to extend the factorization of subdivisions of cone complexes to generalized cone complexes. We do it by a reduction step using barycentric subdivisions:
\begin{lemma}\label{Lem:barycentric}
\begin{enumerate}
\item (\cite[2.5]{ACP})
The barycentric subdivision $B(\Delta)$ of a generalized cone complex $\Delta$ is a projective subdivision obtained by a sequence of simultaneous star subdivisions. If $\Delta$ is nonsingular then the star subdivisions are smooth. The generalized cone complex $B(\Delta)$ is  in fact a cone complex.
\item (\cite[Lemma 8.7]{Abramovich-Matsuki-Rashid}) The barycentric subdivision $B(\Delta)$ of a nonsingular  cone complex $\Delta$ is a projective subdivision obtained by a sequence of simultaneous smooth star subdivisions. The nonsingular  cone complex $B(\Delta)$ is  in fact isomorphic to a fan.
\end{enumerate}
\end{lemma}

\begin{proof}
\begin{enumerate}
\item Write $\Delta =  (\{\sigma_j\},\{\mu_k\})$. We need to show that if $\tau_B,\sigma_B$ are cones in $B(\Delta)$, then a face map $\tau_B\to \sigma_B$ in $B(\Delta)$ is unique if it exists. Suppose the minimal cone containing the image of $\tau_B$ is $\tau$ and the corresponding cone for $\sigma_B$ is $\sigma$. Then it suffices to show that the restriction to $\tau_B$ of  a  face map $\psi:\tau \to \sigma$ in $\Delta$ carrying $\tau_B$ into $\sigma_B$ is unique if it exists. We can write  $\sigma_B = \< b(\sigma_{i_1}),\ldots b(\sigma_{i_k})\>$ uniquely as the cone generated by the barycenters $b(\sigma_{i_r})$ of faces $\sigma_{i_r}$ of $\sigma$ of dimensions $i_1<\dots<i_k$, and similarly $\tau_B = \< b(\tau_{j_1}),\ldots b(\tau_{j_l})\>$. So $\psi$ must carry $b(\tau_{j_s})$ to the barycenter of a cone of $\sigma$ of dimension $j_s$, in other words $\psi(b(\tau_{j_s})) = b(\sigma_{j_s})$. Since $\{b(\tau_{j_1}),\ldots,b(\tau_{j_l})\}$ span $\tau_B$ this means that the restriction of $\psi$ is unique if it exists.
\item Consider the vector space $V = \bigoplus_{\sigma\in \Delta} \RR_\sigma$ with one basis element for each cone of $\sigma$. Assume $\Delta$ is a cone complex. In \cite[Lemma 8.7]{Abramovich-Matsuki-Rashid} it is shown that $B(\Delta)$ has a real embedding in $V$, and the image is the real support of a fan. The embedding is obtained by sending $b(\sigma)$ to the unit vector $e_\sigma\in \RR_\sigma \subset V$. Here we assume that $\Delta$ is nonsingular, and we need to check that the embedding gives an isomorphism of cone complexes, namely that the integral structures coincide. Note that the lattice in any cone $\< b(\sigma_{i_1}),\ldots, b(\sigma_{i_k})\>$  in $B(\Delta)$ is generated by  the elements $ b(\sigma_{i_1}),\ldots ,b(\sigma_{i_k})$. The image of this lattice in $V$ is precisely generated by $e(\sigma_{i_1}),\ldots, e(\sigma_{i_k})$, and coincides with the intersection of the cone $\<e(\sigma_{i_1}),\ldots, e(\sigma_{i_k})\>$ with $\bigoplus_{\sigma\in \Delta} \ZZ_\sigma$. So the image of $B(\Delta)$ is indeed a fan, as required.
\end{enumerate}
\end{proof}

\begin{lemma}\label{Lem:generalized-factorization} Let $\Delta$ be a nonsingular generalized cone complex and  $ f: \Delta \to \RR$ a  piecewise linear  function, convex and integral on each cone, such that the corresponding subdivision  $\Delta_1 \to \Delta$ is nonsingular. Then $\Delta_1 \to \Delta$  admits  a factorization into nonsingular star subdivisions and their inverses, with all intermediate steps projective over $\Delta$.
\end{lemma}
\begin{proof}
By Lemma \ref{Lem:barycentric} we may replace $\Delta_1$ by its second barycentric subdivision, so we may assume $\Delta_1$ is isomorphic to a fan. The common subdivision of $B(B(\Delta_1))$ and  $B(B(\Delta))$ is a projective subdivision of $B(B(\Delta_1))$, so there is a sequence of star subdivisions $\Delta_1'\to B(B(\Delta_1))$ such that $\Delta_1' \to \Delta$ factors through a projective subdivision $\Delta_1' \to \Delta' :=B(B(\Delta))$. Since $\Delta'$ is isomorphic to a fan and $\Delta_1'$ is a projective subdivision, Morelli's $\pi$ desingularization lemma applies, see \cite{Morelli} or \cite[Lemma 10.4.3]{W-Factor}, giving a factorization by star subdivisions and their inverses, all projective over $\Delta'$. Combining these transformation, we obtain the desired factorization, with all steps projective over $\Delta$:

$$\xymatrix{
&&\Delta_1'\ar[dl]|-{\text{star subdivision sequence}}\ar@{.>}[rr]^{\text{factorized}}&&\Delta'\ar@{=}[dr] \\
&B(B(\Delta_1))\ar[dl]|-{\text{star subdivision sequence}}&&&&B(B(\Delta))\ar[dr]|-{\text{star subdivision sequence}}\\
\Delta_1 \ar[rrrrrr]_{\text{projective subdivision}} &&&&&&\Delta
}$$

\end{proof}

\subsection{Functoriality for generalized cone complexes}

\begin{lemma}\label{Lem:generalized-factorization-functorial} The factorization in Lemma \ref{Lem:generalized-factorization} can be made functorial for surjective face maps: we can associate to $(\Delta,f)$ a factorization so that, given a surjective face map $\phi:\Sigma\to \Delta$, the factorization of $(\Sigma,f\circ \phi)$ is the pullback of the factorization of $(\Delta,f)$ along $\phi$.
\end{lemma}
\begin{proof}
For each connected component of the category of pairs $(\Delta,f)$ with face maps between them choose  a final object $(\tilde\Delta,\tilde f)$. By Lemma  \ref{Lem:generalized-factorization} there is a factorization $\tilde\Delta_1 \das \dots\das \tilde\Delta$  of $(\tilde\Delta,\tilde f)$. Given an arbitrary  $(\Delta,f)$ it has a morphism $\psi_{\Delta}: \Delta\to \tilde\Delta$ to the final object   $(\tilde\Delta,\tilde f)$, so that $f = f\circ\psi_\Delta$. The pullback $\Delta_1 \das \dots\das \Delta$  of $\tilde\Delta_1 \das \dots\das \tilde\Delta$ along $\psi_\Delta$ is a factorization of $(\Delta,f)$, and its pullback along $\phi$  is simply the pullback $\Sigma_1 \das \dots\das \Sigma$   along $\psi_\Delta\circ \phi = \psi_\Sigma$ of $\tilde\Delta_1 \das \dots\das \tilde\Delta$, so the process is functorial.
\end{proof}

\subsection{Functoriality for toroidal factorization}
\begin{proof}[Proof of Proposition \ref{Prop:toroidal-factorization}]  The toroidal morphism $X_1 \to X_2$ corresponds to a subdivision $\Sigma(X_1) \to \Sigma(X_2)$ induced by a piecewise linear  function   $f:\Sigma(X_2)\to \RR$ convex and integral on each cone.  This is functorial: a surjective regular morphism $X_2' \to X_2$ gives rise to a surjective face map $\phi:\Sigma(X_2)' \to \Sigma(X_2)$ such that $X_1'\to X_2' $ corresponds to $f\circ \phi$.

By Lemma \ref{Lem:generalized-factorization-functorial} we have a factorization $\Sigma(X_1) \das \dots \das \Sigma(X_2)$, functorial for surjective face maps, into nonsingular star subdivisions and their inverses, with all intermediate steps functorially projective over $\Sigma(X_2)$. This gives rise to a toroidal factorization $X_1 \das \dots \das X_2$ into blowings up and down, which is functorial for surjective regular morphisms, where the terms are functorially projective over $X_2$.

%By Lemma \ref{Lem:final} $f$ is the composition of the final object  $\tilde f: \Delta \to \RR$ of its connected component with the universal morphism $\Sigma(X_2) \to \Delta$. By Lemma \ref{Lem:generalized-factorization} we have a factorization $\Delta_1 \das \cdots \das \Delta$ of the nonsingular subdivision $\Delta_1 \to \Delta$ induced by $\tilde f$ into nonsingular star subdivisions and their inverses, with all intermediate steps projective over $\Delta$. Pulling back along    $\Sigma(X_2) \to \Delta$ we obtain a factorization $\Sigma(X_1) \das \cdots \das \Sigma(X_2)$ into nonsingular star subdivisions and their inverses, with all intermediate steps projective over $ \Sigma(X_2)$. If $\Sigma' \to \Sigma(X_2)$  is any
\end{proof}

 \section{Birational cobordisms}

 %\subsection{Definitions}
 A key tool in the factorization algorithm is the notion of {\em birational cobordism}, introduced in  \cite{W-Cobordism}, where it is motivated by analogy with Morse theory. In this paper we adopt the approach of  \cite{AKMW} which relies on Geometric Invariant Theory and variation of linearizations, see \cite{Brion-Procesi,Thaddeus,Dolgachev-Hu}.

\subsection{Geometric Invariant Theory of $\PP(E)$}\label{Sec:GIT-E} Given a nonzero coherent sheaf $E$ on $X_2$, the data of  a $\GG_m$-action $\rho: \GG_m \to \Aut E$ on $E$  is equivalent to the data of  a $\ZZ$-grading $E = \oplus_{a\in \ZZ} E_a$, which is necessarily a finite sum: $E=\bigoplus_{a=a_{\min}}^{a_{\max}} E_a$.  The homogeneous factor $E_a$ is characterized by $$\rho(t) v\ \  = \ \ t^av \quad \forall v\in E_a.$$ Here and later we use the informal notation $v\in E_a$ to indicate that $v$ is a local section of $E_a$.  Given such data, there is a resulting  action of  $\GG_m$ on $\Sym^\bullet(E)$ and a linearized action on $\PP(E)=\PP_{X_2}(E)$.

We require the following:
\begin{assumption}
 The sheaves  $E_{a_{\min}}$ and $E_{a_{\max}}$ are everywhere nonzero, so $\PP(E_{a_{\min}})\to X_2$ and $\PP(E_{a_{\max}})\to X_2$ are surjective.
 \end{assumption}

 Given an integer $a$ viewed as a character of $\GG_m$, we define a new action of $\GG_m$ on $E$ by $$\rho_a(t) v =  t^{-a}\rho(t)(v).$$ This induces an action on  $\Sym^\bullet(E)$ and on  $(\PP(E),\cO_{\PP(E)}(1))$ which we also denote by $\rho_a$. Writing $ (\Sym^\bullet(E))^{\rho_a}$ for the ring of invariants under this action, we denote $$\PP(E)\sslash_a \GG_m\ \  := \ \ \Proj_{X_2} (\Sym^\bullet(E))^{\rho_a}.$$

As customary, we unwind this as follows: we define the {\em unstable locus of $\rho_a$} to be the closed subscheme
\begin{equation} \label{Eq:unstable}\PP(E)^\un_a \ \ := \ \ \PP\left(\bigoplus_{b<a} E_b\right)\ \ \bigsqcup\ \   \PP\left(\bigoplus_{b>a} E_b\right),
\end{equation} and the {\em semistable locus}  to be the complementary open $$\PP(E)^\sst_a\ \  := \ \ \PP(E)\smallsetminus \PP(E)^\un_a.$$

We have the following well-known facts:
\begin{lemma}\label{Lem:GIT-E}
\begin{enumerate}
\item The semistable locus $\PP(E)^\sst_a$ is nonempty precisely when $a_{\min}\leq a\leq a_{\max}$.
\item Consider the rational map $q_a:\PP(E)\to \PP(E)\sslash_a \GG_m$ induced by the inclusion $(\Sym^\bullet(E))^{\rho_a}\subset (\Sym^\bullet(E))$. Then $q_a$ restricts to an affine $\GG_m$-invariant morphism $\PP(E)^\sst_a  \to  \PP(E)\sslash_a \GG_m$ which is a submersive universal categorical quotient, thus $\PP(E)\sslash_a \GG_m =  \PP(E)^\sst_a\sslash \GG_m.$
\item For $a_{\min}\leq a_1 < a_2\leq a_{\max}$  we have $\PP(E)^\sst_{a_1}\subset\PP(E)^\sst_{a_2}$ precisely when $\bigoplus_{a=a_1}^{a_2-1} E_a = 0$, and similarly $\PP(E)^\sst_{a_1}\supset\PP(E)^\sst_{a_2}$ precisely when $\bigoplus_{a=a_1+1}^{a_2} E_a = 0$. In particular $\PP(E)^\sst_{a_1}=\PP(E)^\sst_{a_2}$ precisely when $\bigoplus_{a=a_1}^{a_2} E_a = 0$.
\item If $a_{\min}\leq a_1 < a_2\leq a_{\max}$  and $\bigoplus_{a=a_1}^{a_2-1} E_a = 0$, then the inclusion $\PP(E)^\sst_{a_1}\subset\PP(E)^\sst_{a_2}$ induces a projective morphism $$\PP(E)^\sst_{a_1}\sslash {\GG_m} \to \PP(E)^\sst_{a_2}\sslash {\GG_m}.$$ Similarly if $\bigoplus_{a=a_1+1}^{a_2} E_a = 0$ we have a projective morphism  $$\PP(E)^\sst_{a_1}\sslash {\GG_m} \leftarrow \PP(E)^\sst_{a_2}\sslash {\GG_m}.$$
\end{enumerate}
\end{lemma}

\begin{proof}
\begin{enumerate}

\item We have  $a\leq a_{\max}$ if and only if $\PP(\bigoplus_{b<a} E_b) \neq \PP(E)$, and  $a_{\min}\leq a$ if and only if $\PP(\oplus_{b>a} E_b)\neq \PP(E)$.

\item \begin{enumerate}

\item {\sc Affine cover of the quotient.}  The scheme $\PP(E)\sslash_a \GG_m =  \Proj_{X_2} (\Sym^\bullet(E))^{\rho_a}$ is covered by principal open sets \begin{equation} \label{Eq:D0} D^0_f :=  (\PP(E)\sslash_a \GG_m)\setminus Z_{\PP(E)\sslash_a \GG_m}(f)\end{equation} associated to non-zero homogeneous invariant elements of the form $f=\prod_{j=1}^s f_j$ where $f_j \in E_{a+\delta_j}$ with $\sum \delta_j = 0$.

 \item {\sc Common zero locus of $\{f\}$.} We note that the common zero locus of elements of $E_c$ is $\PP(E/E_c) = \PP(\bigoplus_{b\neq c} E_b)$. Now observe that any element $f=\prod_{j=1}^s f_j$ as above has a factor $f_j$ with $\delta_j\geq 0$ and a factor $f_j$ with $\delta_j\leq 0$. This means that $f$ vanishes on $\PP(\bigoplus_{b<a} E_b)$ and on $\PP(\bigoplus_{b>a} E_b)$, so $f$ vanishes on $\PP(E)^\un_a$.

Conversely if $x\notin \PP(E)^\un_a$ then  we have some coordinates $f_1\in E_{a+\delta_1}, \delta_1\leq 0$ and $f_2\in E_{a+\delta_2}, \delta_2\geq 0$ which do not vanish: $f_1(x) \neq 0 \neq f_2(x)$.  Taking any positive $r,s$ so that $r\delta_1+ s\delta_2 = 0$ we can form $f = f_1^{r} f_2^s$, and  $f(x) \neq 0$.  This implies that  the common zero locus of the elements $f=\prod_{j=1}^s f_j$ above in $\PP(E)$ is precisely $\PP(E)^\un_a$.

\item {\sc Compatible affine cover of $\PP(E)^\sst_a$.}   It follows that $\PP(E)^\sst_a$ is covered by principal open sets \begin{equation} \label{Eq:D} D_f = \PP(E) \setminus Z_{\PP(E)}(f),\end{equation} the inverse image of the affine open $D^0_f$ of equation (\ref{Eq:D0}) is the affine open $D_f$ of equation (\ref{Eq:D}), and  $\PP(E)^\sst_a \to \PP(E)\sslash_a \GG_m$ is an affine morphism.

\item {\sc Coordinates and invariants} The coordinate ring of $D^0_f$ is  the degree-zero component of $(\Sym^\bullet(E))^{\rho_a}[1/f]$, which is the $\rho_a$-invariant summand of the degree-0 component of $(\Sym^\bullet(E))[1/f]$. The latter is the coordinate ring of $D_f$. In particular, $D^0_f = D_f\sslash \GG_m$ is a submersive universal categorical quotient, see  \cite[Lemma 4.2.6 and Corollary 4.2.11]{ATLuna}. It follows from the definition (see \cite[Remark 5, p. 8]{GIT}) that  $\PP(E)^\sst_a \to \PP(E)\sslash_a \GG_m$ is a submersive universal categorical quotient.
\end{enumerate}
\item The situation is symmetric, so we only address the first statement.  If $\bigoplus_{a=a_1}^{a_2-1} E_a=0$ then $ \PP(\bigoplus_{b<a_2} E_b) =  \PP(\bigoplus_{b<a_1} E_b) \subset \PP(E)^\un_{a_1}$ and certainly $\PP(\bigoplus_{b>a_2} E_b) \subset  \PP(\bigoplus_{b>a_1} E_b)\subset \PP(E)^\un_{a_1}$, so $\PP(E)^\un_{a_1}\subset \PP(E)^\un_{a_2}$ as needed.

Conversely,  if $v\in \PP(\bigoplus_{a=a_1}^{a_2-1} E_a)$ over $x\in X_2$ and we take $w \in \PP(E_{a_{\min}})$ also over $x$, then either $v \in \PP(E_{a_1}) \subset   \PP(E)^\sst_{a_1}$ or  else   $(v+w) \in   \PP(E)^\sst_{a_1}$. In either case, if $\bigoplus_{a=a_1}^{a_2-1} E_a\neq 0$ we have $\PP(E)^\sst_{a_1}\not\subset\PP(E)^\sst_{a_2}$, as needed.
\item The situation is symmetric, so we only address the first case, where $a_{\min}\leq a_1 < a_2\leq a_{\max}$ and $\oplus_{a=a_1}^{a_2-1} E_a = 0$, so that $\PP(E)^\sst_{a_1}\subset\PP(E)^\sst_{a_2}$  by (3). Since   $\PP(E)^\sst_{a_i}  \to  \PP(E)\sslash_{a_i} \GG_m$ are categorical quotients, we have a canonical morphism $\varphi_{a_1/a_2}$ making the following diagram commutative:
$$ \xymatrix{ \PP(E)^\sst_{a_1}\ \  \ar@{^(->}[rr]\ar[d] &&\ \  \PP(E)^\sst_{a_2}\ar[d] \\
\PP(E)\sslash_{a_1} \GG_m\ar[rr]^{\varphi_{a_1/a_2}}&& \PP(E)\sslash_{a_2} \GG_m.
}$$
But $ \PP(E)\sslash_{a_i} \GG_m$ are projective over $X_2$, hence $\varphi_{a_1/a_2}$ is projective.

\end{enumerate}
\end{proof}

This lemma gives the familiar ``wall and chamber decomposition" of the interval $[a_{\min},a_{\max}]$ in the character lattice $\ZZ$ into segments where the quotients $\PP(E)^\sst_{a_1}\sslash_{\GG_m}$ are constant.

All the constructions above are compatible with arbitrary morphisms $X_2' \to X_2$, except that the values of $a_{\min}$ and $a_{\max}$ and the ample sheaf for $\phi_{a_1/a_2}$ are only compatible with surjective morphisms $X_2' \to X_2$.

\begin{remark}\label{Rem:geometric-quotient}
One can show that the quotient morphism $\PP(E)^\sst_a\to   \PP(E)^\sst_a \sslash \GG_m$ is in fact universally submersive. If in addition $E_a=0$ it can be shown that  the quotient morphism is a universal geometric quotient $\PP(E)^\sst_a\to   \PP(E)^\sst_a / \GG_m$. These facts follow from \cite[Theorem 1.1 and Amplification 1.3]{GIT}, which are stated for schemes over a field in characteristic 0 but apply here since $\GG_m$ is a linearly reductive group-scheme over $\ZZ$. Since we do not need these facts, we will not provide a detailed proof, though we will use the notation $\PP(E)^\sst_a / \GG_m$ when $E_a = 0$.
\end{remark}

\subsection{Geometric Invariant Theory of $B\subset \PP(E)$}\label{GIT}
Continuing the discussion, let $B\subset \PP(E)$ be a closed reduced $\GG_m$-stable subscheme. It is the zero locus of  a homogeneous and $\GG_m$-homogeneous ideal $I_B \subset \Sym^\bullet E$.
%We obtain a linearized action $\rho_a$ on $(B, \cO_B(1))$.
We define $B^\un_a := B\cap \PP(E)^\un_a $ and $B^\sst_a := B\cap \PP(E)^\sst_a$. The image of  $q_a:B^\sst_a \to \PP(E)\sslash_a \GG_m$ is denoted $B\sslash_a \GG_m$. We have canonically $B\sslash_a \GG_m=\Proj_{X_2} \left(\left(\Sym^\bullet E/I_B\right)^{\rho_a}\right)$. We write $a_{\min}(B) = \min\{a\mid B\cap \PP(E_a) \neq \emptyset\}$ and similarly $a_{\max}(B) = \max\{a\mid B\cap \PP(E_a) \neq \emptyset\}$. We deduce the analogous, still well-known, facts, which follow immediately from Lemma \ref{Lem:GIT-E}:
\begin{lemma}\label{Lem:GIT-B}
\begin{enumerate}
\item The semistable locus $B^\sst_a$ is nonempty precisely when $a_{\min}(B)\leq a\leq a_{\max}(B)$.
\item The map $q_a:B^\sst_a \to \PP(E)\sslash_a \GG_m$ is an affine $\GG_m$-invariant morphism, inducing  a categorical quotient  $B^\sst_a\to   B^\sst_a\sslash \GG_m=  B\sslash_a \GG_m.$
\item For $a_1 < a_2$  we have $B^\sst_{a_1}\subset B^\sst_{a_2}$ precisely when $B\cap \PP(\oplus_{a=a_1}^{a_2-1} E_a) = \emptyset$, and similarly $B^\sst_{a_1}\supset B^\sst_{a_2}$ precisely when $B\cap \PP(\oplus_{a=a_1+1}^{a_2} E_a) = \emptyset$. In particular $B^\sst_{a_1}=B^\sst_{a_2}$ precisely when $B\cap \PP(\oplus_{a=a_1}^{a_2} E_a) = \emptyset$.
\item If $a_1 < a_2$ and $B\cap \PP(\oplus_{a=a_1}^{a_2-1} E_a) = \emptyset$, then the inclusion $B^\sst_{a_1}\subset B^\sst_{a_2}$ induces a projective morphism $B^\sst_{a_1}\sslash {\GG_m} \to B^\sst_{a_2}\sslash {\GG_m}$. Similarly if $B\cap \PP(\oplus_{a=a_1+1}^{a_2} E_a) = \emptyset$  we have a projective morphism  $B^\sst_{a_1}\sslash {\GG_m} \leftarrow B^\sst_{a_2}\sslash {\GG_m}$.
\end{enumerate}
\end{lemma}
This time we obtain a  ``wall and chamber decomposition" of the interval $[a_{\min}(B),a_{\max}(B)]$. We denote the ``walls", namely the values of $a$ for which $B\cap \PP(E_a) \neq \emptyset$, by $a_{\min}(B) = a_0 < a_1 \dots < a_m = a_{\max}(B)$.

By replacing the embedding $B \subset \PP(E)$ by the Veronese re-embedding $B \subset \PP(\Sym^2E)$ we may, and will, assume
\begin{assumption}\label{Ass:separate}$a_i+1<a_{i+1}$.
\end{assumption}
We denote $B^\sst_{a_i+} = B^\sst_{a_i+1}$ and $B^\sst_{a_i-} = B^\sst_{a_i-1}$, and note that $B^\sst_{a_i+} = B^\sst_{a_{i+1}-}$.  Assumption \ref{Ass:separate} implies that now we always have projective morphisms $\varphi_{a_i\pm}$:
\begin{equation}\xymatrix@=.7pc{B^\sst_{a_i-} /\GG_m \ar[dr]_{\varphi_{a_i-}}\ar@{-->}[rr]^{\varphi_i} &&   B^\sst_{a_i+} /\GG_m\ar[dl]^{\varphi_{a_i+}}  \ar@{=}[r]&  B^\sst_{a_{i+1}-} /\GG_m \ar[dr]^{\varphi_{a_{i+1}-}} & \dots \\
&B^\sst_{a_i} \sslash\GG_m &&& B^\sst_{a_{i+1}} \sslash\GG_m }.
\end{equation}

Finally, we will assume the following:
\begin{assumption}\label{Ass:spread} Each irreducible component of $B$ meets both $\PP(E_{a_{\min}(B)})$ and $\PP(E_{a_{\max}(B)})$.
\end{assumption}
Under this assumption the quotients $B^\sst_{a} \sslash \GG_m$ are all birational to each other, as long as $a_{\min}(B)< a< a_{\max}(B)$. For the extreme values we have isomorphisms $B\cap \PP(E_{a_{\min}(B)}) \to B^\sst_{a_{\min}(B)} \sslash \GG_m$ and $B\cap \PP(E_{a_{\max}(B)}) \to B^\sst_{a_{\max}(B)} \sslash \GG_m$.

\begin{remark}
As in Remark \ref{Rem:geometric-quotient}, it can be shown that $B^\sst_a\to   B^\sst_a \sslash \GG_m$ is universally submersive, and if $B\cap \PP(E_a)=\emptyset$ we have a universal geometric quotient $B^\sst_a\to   B^\sst_a / \GG_m$.
\end{remark}

\subsection{Definition of a birational cobordism}\label{Sec:def-cobordism} The notion of a birational cobordism for a blowing up we use in this paper extends the notion of {\em compactified relatively projective embedded birational cobordism of  \cite[2.4]{AKMW}} by allowing a non-empty boundary. Ignoring the issue of the boundary, it is far more restrictive than the notion introduced in \cite{W-Cobordism}.

Let $\phi\: X_1 \to X_2$ be an object of $\Bl$. A birational cobordism for $\phi$ is a  scheme $B$ which is the blowing up of a $\GG_m$-invariant ideal on $\PP^1_{X_2}$, and embedded, in a manner satisfying Assumptions \ref{Ass:separate} and \ref{Ass:spread}, as a $\GG_m$-stable subscheme in $\PP(E)$ for a $\GG_m$-sheaf $E$ on $X_2$, such that
\begin{enumerate}
\item $X'_1=B^\sst_{a_0+} / \GG_m=B^\sst_{a_0} \sslash \GG_m$ is obtained from $X_1$ by principalizing $D_1$,
\item $X'_2=B^\sst_{a_m-} / \GG_m=B^\sst_{a_m} \sslash \GG_m$ is obtained from $X_2$ by principalizing $D_2$, and
\item the following diagram of rational maps commutes:

$$\xymatrix{
B^\sst_{a_0} \ar[r]^{q_{a_0}}\ar@{-->}[d]_\alpha & X'_1\ar[r] & X_1\ar[d]^\phi\\
B^\sst_{a_m}\ar[r]^{q_{a_m}} & X'_2\ar[r]& X_2}$$
where $\alpha$ is the birational map induced by the open dense inclusions $$B^\sst_{a_0}\subset B\supset B^\sst_{a_m}.$$
\end{enumerate}
The birational cobordism is said to respect the open set $U\subset X_2$ if $U$ is contained in the image of $(B^\sst_{a_0+}\cap B^\sst_{a_m-})/\GG_m$. This happens whenever the ideal on $\PP^1_{X_2}$ whose blowing up is $B$ restricts to the unit ideal on $\PP^1_U$. We say that a birational cobordism $B$ of $\phi$ is {\em regular} if $B$ is regular and the preimage $D_B$ of $D_2$ is a simple normal crossings divisor.

\subsection{Construction of regular birational cobordism}  We claim that one can associate a regular birational cobordism to any blowing up in $\Bl$ functorially, and we formalize this claim as follows. There is an evident category $\Cob$ of regular birational cobordisms of blowings up $\phi\:X_1 \to X_2$ in $\Bl$, with an evident forgetful functor $\Cob \to \Bl$. A morphism of regular birational cobordisms $B' \to B$ is uniquely determined by a regular surjective morphism $g\: X_2'\to X_2$.

 \begin{proposition}\label{Prop:cobordism} The functor $\Cob \to \Bl$ has a section $\Bl \to \Cob$.
 \end{proposition}

We provide a sketch of proof here, and more detail in Appendix \ref{App:cobordism}.
\begin{proof}[Sketch of proof] Following the construction of \cite[Theorem 2.3.1]{AKMW}, consider the blowing up of the ideal $I\otimes \cO_{\PP^1_{X_2}} + I_{\{0\}}$. This is a birational cobordism $B_I$ for $\phi$, but it may be singular. Let $D_{B_I}\subset B_I$ be the preimage of $D_2$. Applying resolution of pairs to $(B_I,D_{B_I})$ we obtain a regular birational cobordism $(B,D_B)$ for $\phi$. Here we use Theorem~\ref{resolth} if the characteristic is zero, and parts (1) and (2) of the Hypothetical Statement \ref{Hyp:resolution} otherwise.
\end{proof}

%\begin{remark}\Michael{I'd remove this remark.}
%We note that the singularities requiring resolution in the second step are far from general. For instance, assume  $I \cO_{X_1}$ is the ideal of a normal crossing divisor. Then if we first blow up $I\otimes \cO_{\PP^1_{X_2}}$, we obtain $\PP^1_{X_1}$, so the ideal $I\otimes \cO_{\PP^1_{X_2}} + I_{\{0\}}$ becomes toroidal and its blowing up has a toroidal resolution of singularities. Thus $B$ admits a projective nonsingular modification (albeit one that modifies the regular locus). Even if the exceptional locus is not a normal crossing divisor, we obtain a modification whose singularities are of the form $st = f(x_1,\ldots,x_k)$, which are rational. This gives some hope  that one might be able to resolve the singularities of $B_I$ even  in positive and mixed characteristic.
%\end{remark}

\section{Factoring the map}
%\subsection{Locally toric structure}
%\begin{lemma}[{\cite[Lemma 2.6.1]{AKMW}}]
%The diagram \ref{Eq:elementary} exhibit $\varphi_i$ as tightly locally toric maps.\Dan{Prove this!}
%\end{lemma}
Throughout this section ``functorial" means ``functorial in $X_1\to X_2$ with respect to surjective regular morphisms". By {\em total transform} of a divisor $D\subset X$ under a (normalized) blowing up $Bl_J(X)\to X$ we mean the union of the preimage of $D$ and the total transform of $J$.

\subsection{Initial factorization}
Proposition \ref{Prop:cobordism} provides a functorial birational cobordism $(B,D_B)$ of $\phi$. Departing slightly from the notation of  \cite[Theorem 2.6.2]{AKMW}, we write $W_{i\pm} = B^\sst_{a_i\pm}/\GG_m $, and  $W_{i} = B^\sst_{a_i}\sslash\GG_m $.  Since $W_{i+}\simeq W_{{(i+1)}-}$ we have a functorial factorization
\begin{equation}\label{Eq:wtd-fact}\xymatrix@=1pc{&& W_{1-}\ar@{=}[ld]_{\varphi_{0+}}\ar[rd]^{\varphi_{1-}}&& W_{2-}\ar[ld]_{\varphi_{1+}}\ar[rd]^{\varphi_{2-}}&&& &W_{m-}\ar@{=}[rd]^{\varphi_{m-}} \ar[ld]_{\varphi_{(m-1)+}}&&
\\
X'_1 \ar@{=}[r]&W_0 && W_1 &&&\ldots&&& W_m \ar@{=}[r] &  X'_2}
 \end{equation}
with all terms functorially projective over $X_2$. Since the cobordism is compatible with $U$, the morphisms $W_{i\pm}\to X_2$ and $W_i\to X_2$ and hence also the morphisms $\varphi_{i\pm}$ are isomorphisms on $U$. Note that since $W_{m-1}\das W_m$ is a morphism it follows that $\varphi_{(m-1)+}$ is an isomorphism, but this fact does not feature in our arguments. In general the terms $W_i$ and $W_i\pm$ in this factorization are singular, but we will use them to construct a non-singular factorization.

\subsection{Blowing up torific ideals}

\subsubsection{Torific ideals}
Let $D_i\subset W_i$, $D_{i\pm}\subset W_{i\pm}$, $D_{a_i}\subset B^\sst_{a_i}$ and $D_{a_i\pm}\subset B^\sst_{a_i\pm}$ denote the preimages of $D_2$. We will show how main results of \cite{AT1} imply that since $(W_i,D_i)$ is given as a quotient of $(B^\sst_{a_i},D_{a_i})$, it can be made toroidal by a canonical torific blowing up. Since $B$ is regular and $D_B$ is a simple normal crossings divisor, $(B^\sst_{a_i},D_{a_i})$ is a toroidal scheme with a relatively affine $\GG_m$-action. In \cite[\S5.4.1]{AT1} one functorially associates to $(B^\sst_{a_i},D_{a_i})$ a $\GG_m$-equivariant {\em normalized torific ideals} $J_i^B$ and $J_i$ on $B^\sst_{a_i}$ and $W_i$, respectively. By abuse of language, the ideal sheaves $J_{i\pm}=J_i\cO_{W_{i\pm}}$ will also be called normalized torific ideals.

\begin{theorem}
For every $1\leq i \leq (m-1)$ the ideal sheaves $J_i$ and $J_{i\pm}$ are functorial and restrict to the unit ideal on $U$. Furthermore, let $W_i^\tor = Bl_{J_i} W_i$ and $W_{i\pm}^\tor = Bl_{J_{i\pm}} W_{i\pm}$, and denote by $D_i^\tor\subset W_i^\tor$ and $D_{i\pm}^\tor\subset W_{i\pm}^\tor$ the total transforms of $D_i$ and $D_{i\pm}$, respectively. Then

\begin{enumerate}
\item $(W_i^\tor,D_i^\tor)$ and  $(W_{i\pm}^\tor,D_{i\pm}^\tor)$ are toroidal, and
\item the morphisms $\varphi_{i\pm}$ induce toroidal morphisms $$\varphi_{i\pm}^\tor: (W_{i\pm}^\tor,D_{i\pm}^\tor) \to (W_i^\tor,D_i^\tor)$$ that restrict to isomorphisms on $U$.
\end{enumerate}
\end{theorem}
\begin{proof}
The ideals $J_i$ are functorial by \cite[Theorem~1.1.2(iii)]{AT1}, hence $J_{i\pm}$ are functorial too. Since the action of $\GG_m$ on $B^\sst_{a_i}$ is already toroidal on $\PP^1_U$, we know by \cite[Theorem~1.1.2(iv)]{AT1} that $J_i$ restrict to the unit ideal of $U$.
%To compute $J_i|_U$ we have to unroll its definition. Recall that $J_i=(I_0^n)^{\rm nor}$, where $n$ is some number, $I$ is the torific ideal on $B^\sst_{a_i}$ associated with the characteristic set $\Sigma^0$ of the toroidal scheme $(B^\sst_{a_i},D^\sst_{a_i})$, and $I_0$ is the $\GG_m$-invariant part of $I$., it follows easily that $\Sigma^0=0$. Therefore, $J_i$ restricts to the unit ideal of $\PP^1_U$, and we obtain that $I_0$ and $J_i$ restrict to the unit ideal of $U$.}

By \cite[Lemma~4.2.12]{AT1} $\GG_m$ acts in a relatively affine way on $B^\tor_{a_i}:=Bl_{J^B_i}(B^\sst_{a_i})$. Let $D^\tor_{a_i}\subset B^\tor_{a_i}$ be the total transform of $D_{a_i}$, then by \cite[Theorem 1.1.2]{AT1}, $(B^\tor_{a_i},D^\tor_{a_i})$ is a toroidal scheme with toroidal action of $\GG_m$, and $W_i^\tor=B^\tor_{a_i} \sslash \GG_m$. Note that $D_i^\tor$ is the image of $D^\tor_{a_i}$, hence $(W_i^\tor,D_i^\tor)$ is toroidal by \cite[Theorem 1.1.3(i)]{AT1}.

By \cite[Lemma~5.5.5]{AT1}, $W_{i\pm}^\tor=(B^\tor_{a_i})_\pm\sslash\GG_m$. Set ($D_{a_i}^\tor)_\pm=D^\tor_{a_i}|_{(B^\tor_{a_i})_\pm}$, then $\GG_m$ acts toroidally on $((B^\tor_{a_i})_\pm,(D_{a_i}^\tor)_\pm)$ and hence the quotient $(W_{i\pm}^\tor,D_{i\pm}^\tor)$ is toroidal by \cite[Theorem 1.1.3(i)]{AT1}. Note also that $\varphi_{i\pm}$ induce toroidal morphisms $\varphi^\tor_{i\pm}$ by \cite[Proposition~5.5.2]{AT1}.
\end{proof}

We note that in general  $W_{i+}^\tor \neq W_{(i+1)-}^\tor.$ The steps  $W_{i-} \to W_i\leftarrow W_{i+}$ in the factorization (\ref{Eq:wtd-fact}) now look as follows:
\begin{equation}\label{Eq:tor-fact}\xymatrix@=.5pc{
& & W_{i-}^\tor\ar[rrd]^{\varphi_{i-}^\tor}\ar[ldd]&&&& W_{i+}^\tor\ar[lld]_{\varphi_{i+}^\tor}\ar[rdd]&\\
& &&& W_i^\tor\ar[dd] &&&\\
W_{(i-1)+}\ar@{=}[r]& W_{i-}\ar[rrrd]^{\varphi_{i-}}&&&&&& W_{i+}\ar[llld]_{\varphi_{i+}}&W_{(i+1)-}\ar@{=}[l]\\
& &&&W_i}
 \end{equation}

\begin{remark}In  \cite[Lemma 3.2.8]{AKMW} it is stated with a sketch of proof that the ideals $J_i$ can be chosen so that $\varphi_{i\pm}^\tor$ are isomorphisms. We will not use this statement. We note however that this follows from   \cite[Theorem 3.5]{Thaddeus}: if the $l$-torific ideal $I_l$ generates all $I_{Ml}, M\geq 1$ and  also $I_{-l}$ generates all $I_{-Ml}, M\geq 1$, then once  $l,-l\in S_i$, the ample set of characters on $B^\sst_{a_i}$ used to determine $J^B_i$ in \cite{AT1}, then  $\varphi_{i\pm}^\tor$ are isomorphisms.  One can choose such $l$ in a manner functorial for regular surjective morphisms.
\end{remark}

\subsection{Resolution and local charts}

\subsubsection{Canonical resolution}\label{Sec:canres}
Extending the notation of \cite[Section 4.2]{AKMW} to qe schemes with a boundary, we write $W_{i\pm}^\res\to W_{i\pm}$ for the resolution of the pair $(W_{i\pm},D_{i\pm})$ and denote the preimage of $D_2$ in $W_{i\pm}^\res$ by $D_{i\pm}^\res$. This morphism is functorially projective and is projectively the identity on $U$. In characteristic zero we use Theorem~\ref{resolth}, and otherwise we invoke Hypothetical Statement \ref{Hyp:resolution}(1). Thus, $W_{i\pm}^\res$ is regular and $D_{i\pm}^\res$ is a simple normal crossings divisor.

Note that the resolution process is independent of the toroidal structures and hence coincides for $(W_{(i-1)+},D_{(i-1)+}) =(W_{i-},D_{i-})$. Thus,  $(W_{(i-1)+}^\res,D_{(i-1)+}^\res) =(W_{i-}^\res,D_{i-}^\res)$ and this provides a bridge between $W_{(i-1)+}^\tor$ and $W_{i-}^\tor$:
$$\xymatrix{W_{(i-1)+}^\tor \ar@{-->}[rr] && W_{(i-1)+}^\res = W_{i-}^\res && W_{i-}^\tor \ar@{-->}[ll]}$$

\begin{remark}\label{princrem}
Since $W_{1-}=X'_1$ is regular, $X''_1:=W^\res_{1-}$ is obtained from $X'_1$ by principalization of $D'_1$ and similarly $X''_2:=W_{m-}^\res$ is obtained from $X'_2$ by principalization of $D'_2$. Both $D'_1$ and $D'_2$ are simple normal crossings divisors, so we could alternatively take $W^\res_{1-}=X'_1$ and $W^\res_{m-}=X'_m$. Our choice above helps to make notation uniform, though it results in a slightly longer factorization.
\end{remark}

\begin{remark} The singularities requiring resolution in this step are far from general: it is shown in the proof of Lemma~\ref{locmonoidlem} below that Zariski locally one can obtain a toroidal scheme from $(W_{i\pm},D_{i\pm})$ simply by enlarging the divisor $D_{i\pm}$. At least over an algebraically closed field they admit resolution of singularities, see \cite[Theorem 8.3.2]{W-Factor}, and it seems reasonable to expect the same in general, and in a functorial manner.
\end{remark}

\subsubsection{Localization}
In Section \ref{Sec:Tying} we will connect $W_{i\pm}^\res$ and $W_{i\pm}^\tor$ by principalizing the ideal $J_{i\pm}^\res:=J_{i\pm}\cO_{W_{i\pm}^\res}$, but to use our principalization conjectures in positive and mixed characteristics we should first check that $J_{i\pm}^\res$ is locally monoidal, so we start with defining local toroidal charts of all our constructions. We will work locally at a point $x\in W_{i\pm}$, so consider the localization $W_x := \Spec\cO_{W_{i\pm},x}$. We denote  $W_x^{\res} = W_{i\pm}^{\res}\times_{W_{i\pm}} W_x$ and similarly for  $W_x^{\tor}$ and other $W_{i\pm}$-schemes we will introduce later. For brevity, we also set $B_x=B_{a_i\pm}^\sst\times_{W_{i\pm}} W_x$, $D_{B_x}=D_{a_i\pm}\times_{B_{a_i\pm}^\sst} W_x$, and $D_x=D_{i\pm}\times_{W_{i\pm}} W_x$. We use the terminology of \cite{ATLuna} regarding strictly local actions and strongly equivariant morphisms, and of \cite{AT1} regarding simple actions and toroidal actions.

\subsubsection{Local toroidal charts}\label{torchart}
The action of $\GGm$ on $B_x$ is simple since $\GGm$ is connected and local since $B_x\sslash\GGm=W_x$. Let $O$ be the closed orbit of $B_x$ and $G_O=\Spec(\ZZ[L_O])$ its stabilizer. Note that $O$ is a torsor under the $k(x)$-group-scheme $\bfD_{K_O}:=\Spec k(x)[K_O]$ with $K_O=\Ker(\ZZ\onto L_O)$. We have two possibilities: (1) $O$ is a point (i.e. the action is strictly local), $G_O=\GGm$, and $L_O=\ZZ$, or (2) the orbit is a torus, $G_O=\mu_n$, and $L_O=\ZZ/n\ZZ$. For a toric monoid $P$ we will use the notation $\bfA_P=\Spec\ZZ[P]$ and $\oE_P=\bfA_P\setminus\bfA_{P^\gp}$. By \cite[Theorem~3.6.11]{AT1} there exists a strongly equivariant strict morphism $h\:(B_x,D_{B_x})\to(\bfA_P,E_P)$, with a suitable $\ZZ$-graded toric monoid of the form $P=\oM_O\oplus K_O\oplus\NN^{\sigma_O}$ and $E_P=\bfA_P\setminus\bfA_{\oM_O^\gp\oplus K_O\oplus\NN^{\sigma_O}}$. Note that the action on $(\bfA_P,E_P)$ is not toroidal, but it becomes toroidal if we enlarge the toroidal structure to $\oE_P$.

\subsubsection{The quotient charts}\label{quotchart}
Let $M=P_0$ be the trivially graded part of $P$. Then $Y:=\bfA_M=\bfA_P\sslash\GGm$ and we consider the divisor $E=E_P\sslash\GGm$ on $Y$, which is a subdivisor of the toroidal divisor $\oE_M=\oE_P\sslash\GGm$. The $\GGm$-action on $(\bfA_P,E_P)$ gives rise to the normalized torific ideal $J_Y$ on $Y$, and let $Y^\tor\to Y$ be the blowing up along $J_Y$. By \cite[Theorem~1.1.2(iii)]{AT1}, the torifications of $Y$ and $W_x$ are compatible with respect to the quotient morphism $h\sslash G\:W_x\to Y$, namely, $J_{W_x}=J_{i\pm}|_{W_x}$ coincides with $J_Y\cO_{W_x}$ and $W_x^{\tor} = W_x\times_Y Y^{\tor}$.

In addition, consider the resolution $Y^\res\to Y$ of the pair $(Y,E)$ as defined in Theorem~\ref{resolth} and Hypothetical Statement \ref{Hyp:resolution}(1). Since the resolution is $\bfA_{M^\gp}$-equivariant, $Y^\res$ is a toric scheme too. Recall that the resolution is compatible with toroidal charts: this follows from the functoriality if $X$ is defined over a field, and we use Hypothetical Statement \ref{Hyp:resolution}(3) in mixed characteristics. Therefore, $W_x^{\res} = W_x\times_Y Y^{\res}$ and the ideal $J_{i\pm}^\res=J_{i\pm}\cO_{W_x^\res}=J_Y\cO_{W_x^\res}$ comes from the ideal $J_Y^\res=J_Y\cO_{Y^\res}$ on $Y^\res$.

%\subsubsection{Enlarged divisors}
\begin{lemma}\label{locmonoidlem}
The ideal $J_{i\pm}^\res$ is locally monoidal.
\end{lemma}
\begin{proof}
We will work locally at $x\in W_{i\pm}$. Let $\oD_{B_x}\subset B_x$ and $\oD_x\subset W_x$ be the preimages of $\oE_P$ and $\oE_M$, respectively. Since $h$ is strongly equivariant, the induced morphism $\oh\:(B_x,\oD_{B_x})\to(\bfA_P,\oE_P)$ is a strongly equivariant toroidal chart. The action on the target of $\oh$ is toroidal, hence the action on the source is toroidal by \cite[Lemma~3.1.9(iv)]{AT1} and $\oh\sslash G\:(W_x,\oD_x)\to(Y,\oE_M)$ is a toroidal chart by \cite[Theorem~1.1.3(iii)]{AT1}.

The resolution $Y^\res\to Y$ is $\bfA_{M^\gp}$-equivariant, hence it is obtained by blowing up a toroidal ideal, and if $\oE^\res$ denotes the total transform of $\oE_M$ then the morphism $(Y^\res,\oE^\res)\to(Y,\oE_M)$ is toroidal. In addition, the pullback of $\oh\sslash G$ gives rise to a toroidal chart $g\:(W_x^\res,\oD^\res_x)\to(Y^\res,\oE^\res)$ with $D^\res_x\subseteq\oD^\res_x$. Since the action on $(Y,\oE_M)$ is toroidal, the ideal $J_Y$ is toroidal with respect to $\oE_M$ by \cite[Lemma~4.4.5(i)]{AT1}. Thus, $J_Y^\res$ is toroidal with respect to $\oE^\res$ and hence its pullback $J_{i\pm}^\res$ is toroidal with respect to $\oD^\res_x$. The lemma follows.
\end{proof}

\subsection{Tying the maps together}\label{Sec:Tying}

\subsubsection{Principalization of torific ideals}
Thanks to Lemma~\ref{locmonoidlem} we can define $W_{i\pm}^{\can}$ to be the canonical principalization of $J_{i\pm}^\res$ in the sense of Section \ref{Sec:principalization}. It is obtained by a {\em functorial sequence of blowings up of nonsingular centers disjoint from $U$ starting from $W_{i\pm}^\res$}, see Proposition \ref{Prop:absolute-principalization}; in positive and mixed characteristics we require Hypothetical Statement \ref{Hyp:principalization}.

By the universal property of blowing up, the maps  $W_{i\pm}^{\can} \das W_{i\pm}^\tor$ are morphisms. The map $W_{i\pm}^{\can} \to W_i$ is a composition of maps given functorially by blowing up ideals restricting to the unit ideal on $U$. By Section \ref{Sec:sequence-projective} the morphism $W_{i\pm}^{\can} \to W_i$ itself is given by blowing up a functorial ideal $\tilJ^\can_{i\pm}$ restricting to the unit ideal on $U$. So, by Lemma~\ref{factorblowlem} the morphism $W_{i\pm}^{\can} \to W_{i\pm}^\tor$ is given by blowing up the functorial ideal $J^\can_{i\pm}=\tilJ^\can_{i\pm}\cO_{W_{i\pm}^\tor}$. By $D_{i\pm}^\can$ we denote the total transform of $D_{i\pm}^\tor$. Diagram (\ref{Eq:tor-fact}) now looks as follows:% In the local setting we also consider the auxiliary toroidal divisor $\oD_x^\can$ which is the total transform of $\oD_x^\tor$

\begin{equation}\label{Eq:can-fact}\xymatrix@=.5pc{
&& W_{i-}^\can\ar[ldd]\ar[dd]&&&& W_{i+}^\can\ar[rdd]\ar[dd]\\ \\
 W_{(i-1)+}^\res\ar@{=}[r]\ar[dd]&W_{i-}^\res\ar[dd] & W_{i-}^\tor\ar[rrd]^{\varphi_{i-}^\tor}\ar[ldd]&&&& W_{i+}^\tor\ar[lld]_{\varphi_{i+}^\tor}\ar[rdd]&W_{i+}^\res\ar[dd]&W_{(i+1)-}^\res\ar@{=}[l]\ar[dd]\\
 &&&& W_i^\tor\ar[dd] &&&\\
 W_{(i-1)+}\ar@{=}[r]& W_{i-}\ar[rrrd]^{\varphi_{i-}}&&&&&& W_{i+}\ar[llld]_{\varphi_{i+}}&W_{(i+1)-}\ar@{=}[l]\\
 &&&&W_i}
 \end{equation}

\begin{lemma}\label{Lem:ideal-is-toroidal}
The ideal $J^\can_{i\pm}$ is toroidal. Thus, $(W_{i\pm}^{\can},D_{i\pm}^\can) \to (W_{i\pm}^\tor,D_{i\pm}^\tor)$ is a functorial toroidal blowing up.
\end{lemma}
\begin{proof}
{\sc Step 1: reduction to toric case.}
We will work locally at $x\in W_{i\pm}$. We already used in \S\ref{quotchart} that torification and resolution are compatible with toroidal charts to show, in the notation introduced there, that $W_x^{\tor} = W_x\times_Y Y^{\tor}$, $W_x^{\res} = W_x\times_Y Y^{\res}$ and $J^\res_x=J_Y^\res\cO_{W_x^\res}$. Let $Y^{\can} \to Y^{\res}$ be the principalization of $J_Y^\res$. Then by the same functoriality argument $W_x^{\can} = W_x\times_Y Y^{\can}$.

By the universal property of blowings up, $Y^\can\to Y$ factors through $Y^\tor$. We have that $Y^\can=Bl_{\tilJ^\can_Y}(Y)$ for a functorial ideal $\tilJ^\can$ on $Y$, hence by Lemma~\ref{factorblowlem}, $Y^\can=Bl_{J^\can_Y}(Y^\tor)$, where $J^\can_Y=\tilJ^\can_Y\cO_{Y^\tor}$. Again, the construction of the ideals $J_\pm^\can$ is compatible with charts. So $J_{i\pm}^\can\cO_{W_x^\can}$ is the pullback of $J_Y^\can$. Thus, it suffices to prove that the ideal $J^\can_Y$ is toroidal.

{\sc Step 2: proof in the toric case.}
In \cite[Proposition 4.2.1]{AKMW} one shows that $(Y^{\can},E^\can)\to(Y^{\tor},E^\tor)$  is toroidal: here we produce this morphism by blowing up the normalized toroidal ideals of \cite{ATLuna} instead of the torific ideal of \cite{AKMW}, but these morphisms have the same equivariance properties. In \cite{AKMW} the ideal blown up is not shown to be toroidal. This can be shown as follows. As in \cite[Proposition 4.2.2]{AKMW} one constructs an action of $\GG_a^k$ on $(Y,E)$. One shows that the morphism $Y^{\tor} \to Y$ of charts is equivariant under this action,  as well as the normalized torific ideal $J_Y$; the scheme  $Y^\tor$  is written as a product of $\GG_a^k$ with a toric scheme providing its toroidal structure. It suffices to show that the ideal defining the blowing up $Y^{\can} \to Y^\tor$  is  a $\GG_a^k$-equivariant monomial ideal, since then its generating monomials are not divisible by the coordinates of the $\GG_a^k$ factor.

Since the blowing up $Y^{\res} \to Y$ is the canonical resolution of singularities of $(Y,E)$, the ideal defining this blowing up on a toric chart is monomial and $\GG_a^k$-equivariant. Also the torific ideal on $Y^{\res}$  is monomial and $\GG_a^k$-equivariant, therefore the same is true for the ideal defining its functorial principalization  $Y^{\can} \to Y^\res$, as required.
\end{proof}

The above lemma implies that the composition $W_{i\pm}^{\can} \to W_{i}^\tor$ is a toroidal morphism given by blowing up a functorial toroidal ideal we denote by $\oJ^{\can}_{i\pm}$. Let $W_i' \to W_i^\tor$ be the normalized blowing up of the product ideal $\oJ^{\can}_{i-}\oJ^{\can}_{i+}$, giving rise to toroidal morphisms $W_i' \to W_{i\pm}^{\can}$.  By \cite[Theorem 3.4.9]{Illusie-Temkin} there is a functorial toroidal resolution  of singularities $W_i^{\tor\res} \to W_i'$. This gives the following:

 %In characteristic 0 the scheme $W_{i}^\tor$ along with the product of these two ideals forms a $\QQ$-absolute pair. By Proposition \ref{Prop:absolute-principalization} in characteristic 0, or  Hypothetical Statement \ref{Hyp:principalization} in general, the pair admits a functorial principalization.\Dan{actually since this is toroidal I think \cite[Section 3.1.14]{Illusie-Temkin} suffices in general!}

\begin{lemma}\label{Lem:torres-construction} There is a toroidal nonsingular modification $W_i^{\tor\res} \to W_{i}^\tor$ obtained by blowing up a functorial ideal, such that the maps $W_i^{\tor\res} \das W_{i\pm}^{\can}$ are both toroidal morphisms.
%\Dan{This is one way, one could instead resolve $W_i^\tor$ and principalize each of the two ideals separately.}
\end{lemma}

Note that these latter maps are again blowings up of the pullbacks of the ideal defining $W_i^{\tor\res} \to W_i^\tor$, which is functorial as well. Since the morphism is toroidal, it induces the identity on $U$, and the toroidal ideal blown up is the unit ideal on $U$.

We now have pieces of the diagram above looking as follows:
$$
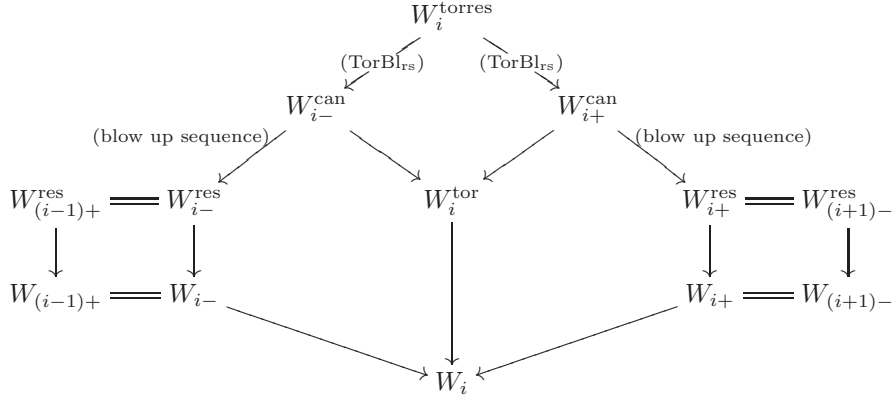
\xymatrix@=1.5pc{
&&& W_i^{\tor\res}\ar[dl]|-{(\TorBl)}\ar[dr]|-{(\TorBl)}\\
&& W_{i-}^{\can}\ar[dl]_{\text{(blow up sequence)}}\ar[dr] && W_{i+}^{\can}\ar[dl]\ar[dr]^{\text{(blow up sequence)}}\\
W_{(i-1)+}^\res\ar@{=}[r]\ar[d]&W_{i-}^\res\ar[d] && W_i^\tor\ar[dd] && W_{i+}^\res\ar[d]&W_{(i+1)-}^\res\ar@{=}[l]\ar[d]\\
 W_{(i-1)+}\ar@{=}[r]&W_{i-}\ar[rrd] &&&& W_{i+}\ar[lld]&W_{(i+1)-}\ar@{=}[l]\\
&&&W_i
}$$
All maps are functorially the blowings up of ideals. The top diamond is at the same time toroidal, with maps given by blowings up of functorial toroidal ideals, so the toroidal structure is functorial in $X_1\to X_2$.  By Proposition \ref{Prop:toroidal-factorization}, the two top maps $W_i^{\tor\res} \to W_{i\pm}^{\can}$ have a functorial toroidal weak factorization; since it is toroidal it induces isomorphisms on $U$. This gives a factorization of the top diamond of the diagram above  as follows:

$$
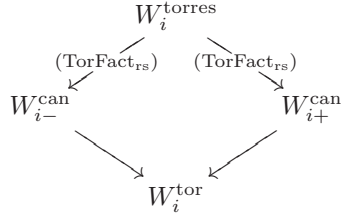
\xymatrix@=1.5pc{
& W_i^{\tor\res}\ar[dl]|-{(\TorFact)}\ar[dr]|-{(\TorFact)}\\
 W_{i-}^{\can}\ar[dr]&& W_{i+}^{\can}\ar[dl]\\
& W_i^\tor
}$$
%$$
%\xymatrix@=1.5pc{
%&&& W_i^{\tor\res}\ar[dl]|-{(\TorFact)}\ar[dr]|-{(\TorFact)}\\
%&& W_{i-}^{\can}\ar[dl]_{\text{(blow up sequence)}}\ar[dr] && W_{i+}^{\can}\ar[dl]\ar[dr]^{\text{(blow up sequence)}}\\
%W_{(i-1)+}^\res\ar@{=}[r]\ar[d]&W_{i-}^\res\ar[d] && W_i^\tor\ar[dd] && W_{i+}^\res\ar[d]&W_{(i+1)-}^\res\ar@{=}[l]\ar[d]\\
% W_{(i-1)+}\ar@{=}[r]&W_{i-}\ar[rrd] &&&& W_{i+}\ar[lld]&W_{(i+1)-}\ar@{=}[l]\\
%&&&W_i
%}$$

Note that $W_{1-}^\res = X''_1$ and $W_{m-}^\res = X''_2$ by Remark~\ref{princrem}. By construction, $X''_i\to X'_i$ and $X'_i\to X_i$ are resolutions of normal crossings pairs $(X'_i,D'_i)$ and $(X_i,D_i)$, respectively, hence $X''_i\to X_i$ factor as sequences of blowings up of regular centers compatible with $U_i$ and $D_i$  thanks to Assumption~\ref{Ass:snc}. Putting these together we functorially obtain a diagram

$$
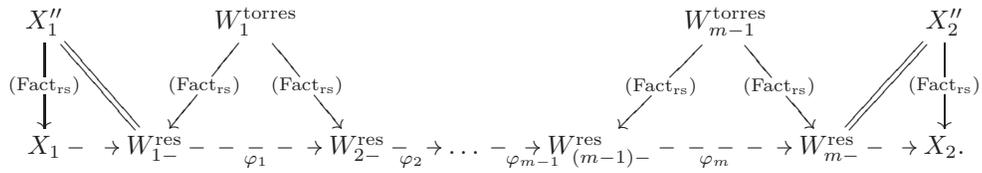
\xymatrix@=0.5pc{
X''_1\ar@{=}[rrddd]\ar[ddd]|-{(\Fact)}&&& W_1^{\tor\res}\ar[lddd]|-{(\Fact)}\ar[dddr]|-{(\Fact)} &&&&&& W_{m-1}^{\tor\res}\ar[dddl]|-{(\Fact)}\ar[dddr]|-{(\Fact)}&&&X''_2\ar@{=}[llddd]\ar[ddd]|-{(\Fact)}
\\ \\ \\
X_1\ar@{-->}[rr]&& W^\res_{1-}\ar@{-->}[rr]_{\varphi_1}&& W^\res_{2-}\ar@{-->}[rr]_{\varphi_2}&&{\ldots\vphantom{M}} \ar@{-->}[rr]_{\varphi_{m-1}} &&W^\res_{(m-1)-} \ar@{-->}[rr]_{\varphi_m}&&  W^\res_{m-}\ar@{-->}[rr]&& X_2.
}$$

Note  that $W_{i}$ are  given by blowing up of functorial ideals on $X_2$, and that $W_{i\pm}^\res$ are obtained by blowing up functorial ideals on $W_i$, all restricting to the identity on $U$. Similarly, the terms appearing in the diagonal arrows are  given by blowing up of functorial ideals on $W_{i\pm}^\res$. By the result of Section \ref{Sec:sequence-projective} all terms appearing are obtained  by blowing up of functorial ideals on $X_2$ restricting to the unit ideal on $U$. In case $X_i\setminus U$ are normal crossings divisors, we have guarantees that the same holds for $W_{i\pm}^\res$. It follows that the same holds for all terms in the sequence forming $W_{i\pm}^\can \to W_{i\pm}^\res$ by the properties of canonical principalization, and for the terms in a factorization of  $W_i^{\tor\res} \to W_{i\pm}^{\can}$ since these are all nonsingular toroidal schemes. Renaming all these terms $V_i, i=1,\ldots,l$, Theorem \ref{Th:main} follows.\qed

\subsection{Summary of resolution steps}

Results around resolution of singulatities were used in several steps in the proof of Theorem \ref{Th:main}. We recall here these steps and what they require. While our main theorem requires the procedures to be functorial, we emphasize the equivariance and functoriality properties necessary for the factorization theorem to hold even without requiring the factorization to be functorial.

The first resolution process appears in the construction of the birational cobordism in Proposition \ref{Prop:cobordism}. This appears explicitly in {\sc Step 3a} in Appendix \ref{App:cobordism}, where we resolve the pair $(B_I, D_{B_I})$, which has dimension $\dim X_2 + 1$. It is crucial that the process be $\GG_m$-equivariant.
%We further require functoriality for the final result to be functorial.

In Section \ref{Sec:canres} we apply resolution of singularities to $W_{i\pm}$, which has dimension $\dim X_2$. The singularities of $W_{i\pm}$ are all locally monomial. Similarly, in Section \ref{Sec:Tying} we apply principalization of the ideals $J_{i\pm}^\res$, which are locally monoidal ideals. On the other hand these two steps require the resolution and principalization to be equivariant in a strong sense: Lemma \ref{Lem:ideal-is-toroidal} requires the process to be compatible with toric charts, and the process on the toric schemes must be both torus equivariant and $\GG_a^k$-equivariant. Finally,   Lemma \ref{Lem:torres-construction} requires toroidal resolution of singularities, which is  as functorial as one could wish.

\section{Extending the factorization to other categories}
In this section we use the factorization for schemes to construct an analogous factorization for blowings up of formal schemes, complex and non-archimedean analytic spaces, and stacks. We follow the general outline of the argument in \cite[Sections 5.1--5.2]{Temkin}, though we decided to elaborate more details related to the relative GAGA issues. In fact, for this construction to work one only needs to have a reasonable comparison theory between algebraic blow ups and their analytifications, but some of these results do not seem to be covered by the literature, especially in the complex analytic case.

\subsection{Stacks}\label{Sec:stacks}
Once functorial factorization for schemes is established it extends to stacks straightforwardly.

\subsubsection{Basic notions}
Our terminology concerning stacks follows that of \cite[\S5.1]{Temkin}. In particular, by a {\em stack} we mean an Artin stack $\fX$ and $\fX$ is {\em qe} (respectively, {\em regular}) if it admits a smooth covering $W\to\fX$ with $W$ a qe (respectively, a regular) scheme. The definition of blowing up along a closed subscheme is compatible with flat morphisms and hence extends to stacks. We define the regular surjective category of blowings up of stacks ${\rm Bl}_{\rm rs}^\fSt$ and the regular surjective category of weak factorizations of blowings up of stacks ${\rm Fact}_{\rm rs}^\fSt$ as in  definitions \ref{Def:Bl} and \ref{Def:factcat}.

\subsubsection{Factorization for stacks}
We are now in position to extend the factorization to stacks.

\begin{theorem}\label{Th:stacks} There is a functor ${\rm Bl}_{\rm rs}^\fSt({\chara = 0}) \to {\rm Fact}_{\rm rs}^\fSt({\chara = 0})$ from the regular surjective category of blowings up $f\:\fX' \to \fX$ in characteristic zero to the regular surjective category of factorizations $$\fX' = \fX_0\das \fX_1\das\dots \das \fX_{l-1} \das \fX_l = \fX,$$ in characteristic zero such that the composite  $${\rm Bl}_{\rm rs}^\fSt({\chara = 0}) \to {\rm Fact}_{\rm rs}^\fSt({\chara = 0})\to {\rm Bl}_{\rm rs}^\fSt({\chara = 0})$$ is the identity. The same holds in positive and mixed characteristics if Hypothetical Statements \ref{Hyp:resolution} and \ref{Hyp:principalization} hold true.
\end{theorem}
\begin{proof}
Choose a smooth covering of $\fX$ by a qe scheme $W$. Then $W$ and $R=W\times_\fX W$ are regular qe schemes and the projections $p_{1,2}\:R\rightrightarrows W$ are surjective and smooth. The pullbacks $W'\to W$ and $R'\to R$ of $\fX'\to \fX$ are objects of $\Bl$, hence Theorem~\ref{Th:main} provides their regular factorizations $(W_{\scriptscriptstyle\bullet})$ and $(R_{\scriptscriptstyle \bullet})$. By the functoriality, these factorizations are compatible with both $p_1$ and $p_2$. Since both pullbacks of the factorization $(W_{\scriptscriptstyle \bullet})$ to $R$ coincide, flat descent implies that $(W_{\scriptscriptstyle \bullet})$ comes from a factorization $(\fX_{\scriptscriptstyle \bullet})$ of $\fX'\to \fX$.

To see that the factorization $(\fX_{\scriptscriptstyle \bullet})$ is independent of a smooth covering $W\to \fX$ we note that any smooth covering $W'\to \fX$ that factors through $W$ induces the same factorization of $\fX'\to \fX$, as follows from the functoriality of factorization with respect to the morphism $W'\to W$.

Finally, assume that $(\fY'\to \fY)\to(\fX'\to \fX)$ is a morphism in ${\rm Bl}_{\rm rs}^\fSt$. Then there exist smooth coverings by qe schemes $W\to \fX$ and $T\to \fY$ such that the morphism $\fY\to \fX$ lifts to a regular surjective morphism $T\to W$. It then follows easily from the functoriality of factorization with respect to $T\to W$ that the factorization for stacks we constructed is compatible with $\fY\to \fX$. Thus, the factorization for stacks is functorial.
\end{proof}

\subsection{Geometric spaces}

\subsubsection{Categories}
We will work with the geometric spaces of the following four classes, that will simply be called {\em spaces}:
\begin{itemize}
\item[(1)] qe formal schemes as defined in \cite[Section 2.4.3]{Temkin},
\item[(2)] semianalytic germs of complex analytic spaces, see Appendix \ref{App:germs},
\item[(3)] $k$-analytic spaces of Berkovich for a complete non-Archimedean field $k$, see \cite[Section 1]{berihes},
\item[(3')] rigid $k$-analytic spaces, where $k$ is as above and non-trivially valued.
%\item[(4)] qe stacks over $\QQ$. \Dan{Phrase conditional statements in positive/mixed characteristic, and in other categories of stacks.}
\end{itemize}

To make notation uniform, the category of all such spaces will be denoted $\fSp$ in each of the four cases.

\begin{remark}
(i) The case (3') is added for the sake of completeness. It is essentially included in (3) because the category of qcqs (i.e. quasi-compact and quasi-separated) rigid spaces is equivalent to the category of compact strictly analytic Berkovich spaces, and all our arguments will be "local enough".

(ii) Probably, there exist other contexts where our methods apply, e.g. semialgebraic geometry. We do not explore this direction here, but we will deal with the above four cases in a uniform way that should make it simpler for the interested reader to extend our results to other possible settings.
\end{remark}

\subsubsection{Affinoid spaces}
We say that a space $X$ is {\em affinoid} if it is of the following type:
\begin{itemize}
\item[(1)] $X=\Spf(A)$ is affine,
\item[(2)] $(\cX,X)$ is an affinoid germ of a complex analytic space, see Section \ref{Sec:affgerm}
\item[(3)] $X=\cM(A)$ is an affinoid $k$-analytic space,
\item[(3')] $X={\rm Sp}(A)$ is an affinoid rigid space over $k$.
\end{itemize}

\subsubsection{Admissible affinoid coverings}
To simplify the discussion we consider only affinoid coverings $X=\cup_{i\in I}X_i$ of a qcqs space by its affinoid domains. Such a covering is called {\em admissible} if it possesses a finite refinement. Here is the main property of admissible coverings, which may fail for non-admissible ones (e.g. the covering of a germ $(\cX,X)$ by one-pointed subgerms $(\cX,x)$ with $x\in X$).

\begin{lemma}\label{Lem:acyclicity}
Assume that $X=\cup_{i\in I}X_i$ is an admissible covering of an affinoid space. Then for any coherent $\cO_X$-module $\cF$ the \c{C}ech complex
$$0\to\cF(X)\to\prod_i\cF(X_i)\to\prod_{i,j}\cF(X_i\cap X_j)\to\dots$$ is acyclic.
\end{lemma}
\begin{proof}
For formal schemes this is classical, and for non-archimedean geometry this is Tate's Acyclicity Theorem and its extension to Berkovich spaces. It remains to deal with complex germs. It suffices to deal with the case of finite coverings, and then we can replace the direct products with  direct sums. Choosing a small enough representative $\cX$ of $X$ we can assume that $\cX$ is Hausdorff. Choose families of Stein domains $V_0\supset V_1\dots$ and $V_{0i}\supset V_{1i}\dots$ for each $i\in I$ such that $X=\cap_{n=0}^\infty V_n$ and $X_i=\cap_{n=0}^\infty V_{ni}$. For each $n\in\NN$ the union $\cup_{i\in I}V_{ni}$ is a neighborhood of $X$ and hence it contains some $V_m$. Let $m=m(n)$ be the minimal number for which the latter happens. The intersections $U_{ni}=V_m\cap V_{ni}$ are Stein domains since $\cX$ is Hausdorff, hence $V_m$ is covered by Stein domains $U_{ni}$ and we obtain the acyclic \c{C}ech complex
$$0\to\cF(V_m)\to\oplus_i\cF(U_{ni})\to\oplus_{i,j}\cF(U_{ni}\cap U_{nj})\to\dots.$$
Since $\lim_{n\to\infty}m(n)=\infty$ and $X_i=\cap_n U_{ni}$, passing to the limit on $n$ we obtain the sequence from the formulation of the Lemma. It remains to use that the filtered colimit is an exact functor.
\end{proof}

\subsubsection{Regular spaces}
Each category of spaces possesses a natural notion of regular spaces, see \cite[Section 5.2.2]{Temkin}. In fact, a space $X$ is regular if it possesses an admissible affinoid covering $X=\cup_iX_i$ such that the rings $A_i=\cO_X(X_i)$ are regular. In particular, it follows from Lemma~\ref{Lem:affgerm} that a germ of analytic space $(\cX,X)$ is regular if and only if $\cX$ is smooth in a neighborhood of $X$.

By $\fSp_\reg$ we denote the full subcategory of $\fSp$ consisting of quasi-compact regular objects, and we do not impose any separatedness assumption.

\subsubsection{Smooth and regular morphisms}
Also, the category $\fSp$ has a natural notion of smooth morphisms. In cases (1), (2) and (3') this is the classical notion (with the obvious adjustment in (2)) and in (3) this is the notion of quasi-smooth morphisms as defined in \cite[Section 4]{Ducros}. %Note that in cases (2) regularity and smoothness of a space coincide, while in (3) and (3') a space if smooth over the ground field $k$ if and only if it is geometrically regular over $k$.

In cases (2), (3) and (3') any morphism is of finite type, so we identify the notions of smooth and regular morphisms. Regular morphisms of qe formal schemes were defined in \cite[2.4.12]{Temkin}: a morphism $f\:Y\to X$ is called {\em regular} if it admits an open covering of the form $f_i\:\Spf(B_i)\to\Spf(A_i)$ such that the homomorphisms $A_i\to B_i$ are regular.

\begin{lemma}\label{Lem:regmor}
If $Y\to X$ is a regular morphism of affinoid spaces in $\fSp$ then the homomorphism $\cO_X(X)\to\cO_Y(Y)$ is regular.
\end{lemma}
\begin{proof}
Case (1) is covered by \cite[Lemma~2.4.6]{Temkin}. Case (3), and hence also case (3'), follows from \cite[Proposition 4.5.1]{Ducros}, \cite[Theorem 3.3]{Ducros-excellent} and the fact that for any affinoid space $Z=\cM(C)$ the map $Z\to\Spec(C)$ is surjective by \cite[Proposition 2.1.1]{berihes}. Case (2) is dealt with similarly using that if $Z$ is an affinoid germ, $z\in Z$ and $f\:Z\to T=\Spec(\cO_Z(Z))$ is the natural map then $f(Z)$ is the set of all closed points and the homomorphism $\cO_{T,f(z)}\to\cO_{Z,z}$ is regular by Lemma~\ref{Lem:affgerm}.
\end{proof}

\subsection{Relative GAGA}
Assume that $X$ is an affinoid space, $A=\cO_X(X)$ and $\cX=\Spec A$. Relative GAGA relates the theory of $\cX$-schemes and $X$-spaces.

\subsubsection{Analytification functor}
There exists an analytification/formal completion functor from $\cX$-schemes of finite type to $X$-spaces. For uniformity, we will usually call this functor {\em analytification} and denote $\cY\mapsto Y=\cY^\an$. It is constructed as follows:
\begin{itemize}
\item[(i)] The analytification of $\AA^n_\cX$ is $\AA^n_X$.
\item[(ii)] If $\cY$ is $\cX$-affine, say $\cY=\Spec B$ with $B=A[t_1\.t_n]/(f_1\.f_m)$, then $\cY^\an$ is the vanishing locus of $f_1\.f_m$ in $\AA^n_X$. It is easily seen to be independent of the $A$-presentation of $B$
\item[(iii)] The construction in (ii) is compatible with localizations, so in general one covers $\cY$ by $\cX$-affine schemes $\cY_i$ and glues $\cY^\an$ from $\cY_i^\an$.
\end{itemize}

\subsubsection{The analytification map}
There exist natural {\em analytification maps} $\pi_\cY\:\cY^\an\to\cY$ which can be constructed through the steps (i)--(iii), or directly (ii) and (iii). Let us describe them in the affine case $\cY=\Spec B$:
\begin{itemize}
\item[(1)] The map is $\Spf B\into\Spec B$. It is injective and the image is the set of open prime ideals of $B$.
\item[(2),(3')] The map $\cY^\an\to\cY$ is injective and its image is the set of maximal ideals of $B$.
\item[(3)] The map $\cY^\an\to\cY$ is surjective, see \cite[Proposition 2.6.2]{berihes}.
\end{itemize}

\subsubsection{Sheaves}
The analytification functor also extends to coherent sheaves: for any $\cX$-scheme $\cY$ of finite type there exists an analytification functor $\Coh(\cY)\to\Coh(\cY^\an)$ given by $\cF^\an=\pi_\cY^*\cF$.

\subsubsection{Properties}\label{Sec:Gagaproperties}
For each $\cX$-proper scheme $\cY$ the analytification functor $\Coh(\cY)\toisom\Coh(Y)$ is an equivalence of categories. In particular, the analytification functor induces an equivalence between the categories of projective $\cX$-schemes and $X$-spaces. The references are:
\begin{itemize}
\item[(1)] Grothendieck's Existence Theorem, \cite[III$_1$, 5.1.4]{ega}.
\item[(2)] Theorem  \ref{Th:GAGA} below.
\item[(3)] The analytification was introduced in \cite[Section 2.6]{berihes}, and comparison of coherent sheaves can be found in \cite[Theorem A.1]{Poineau}.
\item[(3')] K\"opf's theorem, see \cite[Sections 5 and 6]{Kopf} and \cite[Example 3.2.6]{relativeample}.
\end{itemize}

\subsubsection{Analytification and regularity}
Various properties are respected by analytification, but for our needs we only need to study the situation with regularity.

\begin{proposition}\label{Prop:gagareg}
Assume that $X$ is an affinoid space with $A=\cO_X(X)$, $\cX=\Spec(A)$, and $\cY$ is an $\cX$-scheme of finite type with $Y=\cY^\an$, then
\begin{itemize}
\item[(i)] If $\cY$ is regular then $Y$ is regular.

\item[(ii)] Conversely, assume that $Y$ is regular, then

\subitem(a) in cases (2), (3) and (3'), $\cY$ is regular,

\subitem(b) in case (1) assume also that $\cY$ is $\cX$-proper, then $\cY$ is regular.
\end{itemize}
\end{proposition}
\begin{proof}
Note that case (3') follows from (3) since a qcqs rigid space can be enhanced to an analytic space, and the regularity is preserved. We will study cases (1), (2) and (3) separately, but let us first make a general remark. The claims (i) and (ii)(a) are local on $\cY$, so we can assume that $\cY=\Spec B$ for a finitely generated $A$-algebra $B$ in these cases.

Case (1). In this case, $A$ is an $I$-adic ring and $X=\Spf A$. Since $A$ is qe, $B$ is qe and so the $I$-adic completion homomorphism $B\to\hatB$ is regular. This implies (i) since if $B$ is regular then $\hatB$ is regular, and so $\Spf\hatB$ is regular.

Let us prove (ii). Since $A$ is $I$-adic, $I$ is contained in the Jacobson radical of $A$ (see \cite[Proposition~10.15(iv)]{AM}), and so any point of $\cX$ has a specialization in $\cX_s:=V(I)$. By the properness of $f\:\cY\to\cX$, any point of $\cY$ has a specialization in $\cY_s:=f^{-1}(\cX_s)$, hence it suffices to prove the following claim: if $\cY$ is of finite type over $\cX$ and $Y$ is regular, then $\cY$ is regular at any point $y\in\cY_s$.

The latter claim is local around $y$, hence we can assume, again, that $\cY=\Spec B$. Let $m\subset B$ be the ideal corresponding to $y$, then the $m$-adic completion $B\to\hatB_m$ factors through the $I$-adic completion $B\to\hatB$, and so $\hatB_m$ is the completion of $\hatB$ along $m\hatB$. Since $X$ is qe, $\hatB$ is qe and so $\hatB\to\hatB_m$ is regular. By our assumption $\hatB$ is regular, hence $\hatB_m$ is regular too. The homomorphism $B_m\to\hatB_m$ is faithfully flat, hence $B_m$ is regular and we win.

Case (3). In this case, $A$ is $k$-affinoid and $X=\cM(A)$. Consider a point $y\in Y$ and set $\by=\pi_\cY(y)\in\cY$. By \cite[(1.3.7.2)]{Ducros}, $\cY$ is regular at $\by$ if and only if $Y$ is regular at $y$. Since $\pi_\cY$ is surjective this implies that $\cY$ is regular if and only if $Y$ is so.

Case (2). If $y\in Y$ and $\by=\pi_\cY(y)$ then it follows easily from Lemma~\ref{Lem:affgerm} that the homomorphism $f_y\:\cO_{\cY,\by}\to\cO_{Y,y}$ induces an isomorphism of the completions. A local ring is regular if and only if its completion is regular, hence $\cO_{\cY,\by}$ is regular if and only if $\cO_{Y,y}$ is so. Since the image of $\pi_\cY$ contains all closed points, we obtain that $Y$ is regular if and only if $\cY$ is regular.
\end{proof}

\subsection{The factorization theorem}
\subsubsection{Blowings up}
Each of the categories $\fSp$ has a natural notion of blowings up $f\:X'\to X$ along ideals (e.g., see \cite[Section 2.4.4]{Temkin} and \cite[Section 5.1.2]{Temkin}). In fact, $Bl_I(X)$ can be described as follows: if $Y\subset X$ is an affinoid domain, $\cY=\Spec(\cO_X(Y))$ and $\cI\subset\cO_\cY$ is induced by $I$, then the restriction of $f$ onto $Y$ is the analytification of the blowing up $Bl_\cI(\cY)\to\cY$. We will only consider blowings up with nowhere-dense centers.

\subsubsection{Weak factorization}
By a weak factorization of $X_1 \to X_2$ we mean a diagram
$$\xymatrix{X_1 = V_0\ar@{-->}[r]^(.6){\phi_1}& V_1\ar@{-->}[r]^{\phi_2}&\ldots \ar@{-->}[r]^{\phi_{l-1}} &V_{l-1} \ar@{-->}[r]^(.4){\phi_l}&  V_l = X_2}$$
along with subspaces $Z_i$ and ideal sheaves $J_i$  satisfying conditions (1-5) of Section \ref{Sec:def-factor}, where in (2) and (4) the word ``scheme" is replaced with ``space". For brevity of notation, such a datum will be denoted $(V_{\scriptscriptstyle \bullet},\phi_{\scriptscriptstyle \bullet},Z_{\scriptscriptstyle \bullet})$.

We define the regular surjective category of blowings up ${\rm Bl}_{\rm rs}^\fSp$ in $\fSp$ and the regular surjective category of weak factorizations ${\rm Fact}_{\rm rs}^\fSp$ on $\fSp$ as in Definitions \ref{Def:Bl} and \ref{Def:factcat}. By definition, these categories are fibred over the category of regular spaces with regular morphisms, and the fibers over a regular space $X$ will be denoted $\Bl(X)$ and $\Fact(X)$. Thus, $\Bl(X)$ is the set of blowings up $X'\to X$ with regular $X$ and $\Fact(X)$ is the set of all regular factorizations of blowings up of $X$.

\begin{lemma} \label{Lem:gagafact}
Let $X$ be an affinoid space, $A=\cO_X(X)$ and $\cX=\Spec A$. Then the analytification functor $\cY\mapsto\cY^\an$ induces bijections $\Bl(X)\toisom\Bl(\cX)$ and $\Fact(X)\toisom\Fact(\cX)$.
\end{lemma}
\begin{proof}
By the relative GAGA, see Section \ref{Sec:Gagaproperties}, analytification induces a bijection between the blowings up $X'\to X$ and $\cX'\to\cX$. By Proposition~\ref{Prop:gagareg}, $X'$ is regular if and only if $\cX'$ is regular, hence $\Bl(X)\toisom\Bl(\cX)$. The second bijection is proved similarly, but this time one also relates regularity of the centers in the factorizations.
\end{proof}

\subsubsection{The main theorem}
We are now in position to prove the following analogue of Theorem~\ref{Th:main}.

\begin{theorem}\label{Th:main-C} There is a functor ${\rm Bl}_{\rm rs}^\fSp({\chara = 0}) \to {\rm Fact}_{\rm rs}^\fSp({\chara = 0})$ from the regular surjective category of blowings up $f\:X' \to X$ in characteristic zero to the regular surjective category of factorizations $$X' = V_0\das V_1\das\dots \das V_{l-1} \das V_l = X,$$ in characteristic zero such that the composite  $${\rm Bl}_{\rm rs}^\fSp({\chara = 0}) \to {\rm Fact}_{\rm rs}^\fSp({\chara = 0})\to {\rm Bl}_{\rm rs}^\fSp({\chara = 0})$$ is the identity. The same holds in positive and mixed characteristics if Hypothetical Statements \ref{Hyp:resolution} and \ref{Hyp:principalization} hold true.
\end{theorem}
\begin{proof}
%We start with the cases (1), (2), (3) and (3').
First, let us construct a factorization of $f\:X'\to X$. Fix an admissible affinoid covering $X=\cup_{i=1}^n X_i$ and set $X'_i=X_i\times_XX'$. The rings $A_i=\cO_X(X_i)$ are qe, see \cite[Section 5.2.3]{Temkin}, so the scheme $\cX=\coprod_{i=1}^n\cX_i$ with $\cX_i=\Spec(A_i)$ is noetherian and qe. Let $I$ be the ideal defining $f$ and let $I_i\subset A_i$ be its restrictions. Consider the blowings up $F_i\:\cX'_i\to\cX_i$ defined by $I_i$. The analytification of $F_i$ is the restriction $f_i$ of $f$ over $X_i$ by the relative GAGA, hence $\cX'_i$ is regular by Proposition~\ref{Prop:gagareg}(ii).

Set $\cX'=\coprod_{i=1}^n\cX'_i$ and consider the factorization $(\cV_{\scriptscriptstyle \bullet},\Phi_{\scriptscriptstyle \bullet},\cZ_{\scriptscriptstyle \bullet})$ of the blow up $F\:\cX'\to\cX$. For each $i$, it induces a factorization
$(\cV_{i,{\scriptscriptstyle \bullet}},\Phi_{i,{\scriptscriptstyle \bullet}},\cZ_{i,{\scriptscriptstyle \bullet}})$ of $F_i\:\cX'_i\to\cX_i$ and the analytification of the latter is a factorization of $f_i\:X'_i\to X_i$ that will be denoted $(V_{i,{\scriptscriptstyle \bullet}},\phi_{i,{\scriptscriptstyle \bullet}},Z_{i,{\scriptscriptstyle \bullet}})$.

We claim that the latter factorizations glue to a factorization of $f$. It suffices to prove that for any $i,j$ and an affinoid domain $Y\subset X_i\cap X_j$ the restrictions of $(V_{i,{\scriptscriptstyle \bullet}},\phi_{i,{\scriptscriptstyle \bullet}},Z_{i,{\scriptscriptstyle \bullet}})$ and $(V_{j,{\scriptscriptstyle \bullet}},\phi_{j,{\scriptscriptstyle \bullet}},Z_{j,{\scriptscriptstyle \bullet}})$ onto $Y$ coincide. Set $B=\cO_X(Y)$ and $\cY=\Spec(B)$, and let $G\:\cY'\to\cY$ be the blowing up along the ideal induced by $I$. In particular, the analytification $g\:Y'\to Y$ of $G$ is the restriction of $f$. The regular homomorphisms $A_i\to B$ and $A_j\to B$ induce regular morphisms $h_i,h_j\:\cY\to\cX$ such that $G$ is the pullback of $F$ with respect to either of this morphisms. The factorizations of $G$ induced from $(\cV_{\scriptscriptstyle \bullet},\Phi_{\scriptscriptstyle \bullet},\cZ_{\scriptscriptstyle \bullet})$ via $h_i$ and $h_j$ coincide by Lemma~\ref{Lem:func} below. It remains to note that the factorizations of $g$ induced from the factorizations of $f_i$ and $f_j$ are the analytifications of these factorizations of $G$.

We have constructed a factorization of $f$. The same argument as was used to glue local factorizations to a global one shows that the construction is independent of the affinoid covering. Finally, compatibility of factorization with a regular morphism $h\:Y\to X$ is deduced in the same way from Lemma~\ref{Lem:regmor} and compatibility with regular morphisms of factorization for schemes.
\end{proof}

The following result is an analogue of \cite[Lemma~2.3.1]{Temkin}.

\begin{lemma}\label{Lem:func}
Assume that $\cF\:\Bl \to \Fact$ is a factorization functor, $f\:X'\to X$ and $g\:Y'\to Y$ are two blowings up with regular source and target and $h_i\:Y\to X$ with $i=1,2$ are two regular morphisms such that $h_i^*(f)=g$. Then the pullbacks of $\cF(f)$ to a factorization of $g$ via $h_1$ and $h_2$ coincide.
\end{lemma}
\begin{proof}
Extend $h_i$ to morphisms $\phi_i\:Y\coprod X\to X$ so that the map on $X$ is the identity. Each $\phi_i$ is a surjective regular morphism, hence the pullback of $\cF(f)$ to $Y\coprod X$ via $\phi_i$ coincides with the factorization of the blowing up $Y'\coprod X'\to Y\coprod X$. Restricting the latter onto $Y$ coincides with $h_i^*(\cF(f))$.
\end{proof}

\begin{remark}
(i) An analogue of Lemma~\ref{Lem:func} holds true in any category $\fSp$ and the above proof applies verbatim.

(ii) Although $h_i^*(\cF(f))$ coincide, they can differ from $\cF(g)$ when $h_i$ are not surjective. See also \cite[Remark~2.3.2(ii)]{Temkin}.
\end{remark}

\appendix

\section{Construction of a birational cobordism via deformation to the normal cone}\label{App:cobordism}

\begin{proof}[Proof of Proposition \ref{Prop:cobordism}]
 We follow the construction of \cite[Theorem 2.3.1]{AKMW} word for word, except we make it even more explicit and check functoriality.

 {\sc Step 1: \em cobordism $B_\cO$ for trivial blowing up.} We start with $$B_\cO=\PP^1_{X_2} = \PP(\cO_{X_2}\cdot T_0 \oplus \cO_{X_2}\cdot T_1) =: \PP_{X_2}(E_\cO),$$ with  its projection $\pi_0: B_\cO \to X_2$. Providing the  generators $T_0$ and $T_1$ with $\GG_m$-weights 0 and 1, the scheme $B_\cO$ is a birational cobordism for the identity morphism with the trivial  ideal $(1)$, with the standard action of $\GG_m$ linearized, except that it does not satisfy Assumption \ref{Ass:separate}. But that may be achieved after the fact by taking the symmetric square.
The construction is clearly functorial.

 {\sc Step 2a: \em construction of a singular cobordism $B_I$.}  Assume $X_1$ is given as the blowing up of the ideal $I$ on $X_2$. We blow up the $\GG_m$-equivariant ideal $I^B:=I\otimes \cO_{B_\cO} + I_{\{0\}}$ on $B_\cO$, where $I_{\{0\}}$ is the defining ideal of  $\{0\}\times X_2$. The ideal is clearly the unit ideal on $\PP^1_U$.  This blowing up gives rise to a $\GG_m$-scheme $B_I$ and projective morphism $\pi_I: B_I\to B_\cO$;  this is evidently functorial in $\phi$. The arguments of Section \ref{Sec:sequence-projective} show that $\pi^{B_I/X_2}:=\pi_0\circ\pi_I:B_I\to X_2$ is  projective, again in a functorial manner. In particular $B_I \subset \PP(E_I)$ for some functorial $\GG_m$-sheaf $E_I$.
 %\Dan{find a way to write this explicitly! describe $(B_I)^\sst_{a_0}$ and $(B_I)^\sst_{a_m}$ and show they are line bundles on $X_1,X_2$.}

 {\sc Step 2b: \em coordinates of $B_I$.}  Let us make the construction of the previous step explicit: write $F_I = \pi_{0\,*} I^B(1) =  I\cdot U_0\oplus \cO_{X_2}\cdot U_1 $ with $U_0,U_1$ having corresponding $\GG_m$-weights $0$ and $1$.   Let $$E_I \ \  = \ \  F_I \otimes E_\cO \ \ = \ \ I\cdot U_0T_0 \ \oplus \ (\cO_{X_2}\cdot U_1T_0 \oplus I\cdot U_0T_1)\  \oplus \ \cO_{X_2}\cdot U_1T_1$$ with corresponding $\GG_m$-weights $0,1$ and $2$.  Again it does not satisfy Assumption \ref{Ass:separate}, but again that may be achieved after the fact by taking the symmetric square.

We have a surjection $\pi_0^*F_I \to I^B(1)$ where the first coordinate sends $f\cdot U_0\mapsto f T_0$ and the second sends  $U_1\mapsto T_1$. We thus have  $\GG_m$-equivariant closed embeddings
\begin{align*}
B_I =Bl_{I^B}(B_\cO)=Bl_{I^B(1)}(B_\cO)&  \\
 \subset \ \   \PP_{B_\cO}(\pi_0^*F_I) &= \PP_{X_2}(F_I) \times_{X_2} B_\cO = \PP_{X_2}(F_I) \times_{X_2} \PP_{X_2}(E_\cO)\\ &\subset\ \ \  \PP_{X_2}(F_I \otimes E_\cO) = \PP_{X_2}(E_I) ,
\end{align*}
where $Bl_{I^B(1)}(B_\cO)$ denotes the blowing up of the fractional ideal $I^B(1)$ and the last inclusion is the Segre embedding.

We describe $B_I = \Proj_{X_2} A$ as follows. The algebra
\begin{align*}A \ :=&\\   \bigoplus_d &\left(I^d\cdot T_0^{2d} \  \ \oplus\ \  I^{d-1} \cdot T_0^{2d-1}T_1 \ \  \oplus \dots  \oplus \ \ \cO_{X_2}\cdot T_0T_1^{2d-1} \ \ \oplus \ \ \cO_{X_2}\cdot T_1^{2d} \right),
\end{align*}
with terms $I^{d-k}\cdot T_0^j T_1^k$ when $j>k$ and $\cO_{X_2} \cdot T_0^j T_1^k$ when $j\leq k$,
is a graded $\GG_m$-weighted quotient
 $\Sym^\bullet E_I  \onto  A , $ where we set $U_j = T_{j}$ and map $I^{\otimes d} \onto I^d$.

 We note that $B_I$ admits an equivariant  projection morphisms $B_I \to B_\cO = \PP_{X_2}(E_\cO)$ which is an isomorphism away from the divisor  $(T_1^2)$, and an equivariant  projection morphism $B_I \to \PP_{X_2}(F_I)$, whose image is the closed subscheme we denote
$$\PP_{X_2}(F_I)' :=\Proj_{X_2} \bigoplus_{n\geq 0} \left(\bigoplus_{j=0}^n I^j\right).$$
The morphism $B_I \to \PP_{X_2}(F_I) ' $ is an isomorphism away from the zero section $\Proj_{X_2} \bigoplus_{n\geq 0} \cO_{X_2} \subset \PP_{X_2}(F_I)'$, whose complement is the total space $\Spec \Sym((IO_{X_1})^{-1})$  of the invertible sheaf $I \cO_{X_1}$ on $X_1$.

 %relative to $X_2$ as follows. Relative to $B_\cO$ we have $$B_I = \Proj_{B_\cO} \left(\oplus_d \left(I^d\oplus \cdots I \oplus \cO_{B_\cO}\right)\right) / ($$

%Writing $E_I = I \oplus \cO_{X_2}$ we have a natural embedding $B_I\subset \PP(E_I)$, which we compose with the Veronese $\PP(E_I)\hookrightarrow \PP(\Sym^2 (E_I))$ so that . Indeed $B_I = \Proj_{X_2} A_1$ where $$(A_1)_n \ \ =\ \  I^n\oplus I^{n-1}\oplus \dots\oplus I \oplus \cO_{x_2}.$$ The $\GG_m$ action is inherited from the inclusion $E_I \subset E$.

{\sc Step 2c: \em stable and unstable loci for weight $1$.} The homogeneous Cartier divisor $(T_0T_1)$ is the union of two  regular subschemes: $X_1  = \Proj_{X_2} \bigoplus_{n\geq 0} (I^n\cdot T_0^{2n})$  which is the zero locus of $(T_0T_1,T_1^2)$, and  $X_2 = \Proj_{X_2} \bigoplus_{n\geq 0} (\cO_{X_2}\cdot T_1^{2n})$ which is the zero locus of $(T_0T_1, I\cdot T_0^2)$. Since the zero locus of  the ``irrelevant ideal" $(I\cdot T_0^2,T_0T_1,T_1^2)$ is empty, these two subschemes are disjoint. In particular each is a regular Cartier divisor.   It follows that  both $X_1$ and $X_2$ lie in the regular locus $B_I^{\reg}$, which is open since $B_I$ is of finite type over the qe scheme $X_2$.

We have $X_1 =  B_I \cap \PP_{X_2}((E_I)_{0})$ and $X_2 =  B_I \cap \PP_{X_2}((E_I)_{2})$, where the indices $0$ and $2$ denote the components with given $\GG_m$-weight (the variable $a$ in Section \ref{GIT}). Their union $(T_0T_1)$ is the unstable locus $(B_I)^\un_{1}$. The complement is affine, explicitly
\begin{align*} (B_I&)^\sst_{1} = \Spec_{X_2} A[(T_0T_1)^{-1}]_{\text{degree}=0}\\
  =& \Spec_{X_2}\left( \dots \oplus I^2 \left(\frac{T_0}{T_1}\right)^2 \oplus I \left(\frac{T_0}{T_1}\right) \oplus \cO_{X_2} \oplus \cO_{X_2} \left(\frac{T_1}{T_0}\right) \oplus \cO_{X_2} \left(\frac{T_1}{T_0}\right)^2 \oplus \dots  \right).
\end{align*} This scheme is in general singular, but the quotient is simpler: $$(B_I)^\sst_{1} \sslash \GG_m = \Spec_{X_2} \cO_{X_2} = X_2.$$

{\sc Step 2d: \em stable and unstable loci for weight $2$.}
The projective Cartier divisor $(T_1^2)$ can be identified as
\begin{align*}(B_I)^\un_{2} &= \PP_{X_2}(I\cdot T_0^2)\ \cup\ \PP_{Z(I)}(I/I^2\cdot T_0^2 \oplus \cO\cdot T_0T_1)\\ &= X_1\qquad \quad \ \ \ \cup\qquad  C(Z(I)),\end{align*}
where   $C(Z(I))$ is the normal cone. The complement is again affine, of the form
\begin{align*} (B_I)^\sst_{2} =& \Spec_{X_2} A[T_1^{-1}]_{\text{degree}=0}\\
  =& \Spec_{X_2}\left( \dots \oplus \cO_{X_2} \left(\frac{T_0}{T_1}\right)^2 \oplus \cO_{X_2}  \left(\frac{T_0}{T_1}\right) \oplus \cO_{X_2}   \right) =\AA^1_{X_2}.
\end{align*}

Thus, $$(B_I)^\sst_{2} \sslash \GG_m = \Spec_{X_2} \cO_{X_2} = X_2$$ and the morphism $(B_I)^\sst_{2}\to X_2$ is smooth. Another way to see this is to notice that the map $B_I \to B_\cO$ restricts to an open embedding on $(B_I)^\sst_{2}$, and the image is the complement of $\{0\} \times X_2$.

{\sc Step 2e: \em stable and unstable loci for weight $0$.}
The projective zero locus of $(I\cdot T_0)^2$ can be identified as $$(B_I)^\un_{0} = \PP_{X_2}(\cO_{X_2}\cdot T_1^2) \cup \PP_{Z(I)}(\cO_{X_2}\cdot T_0T_1 \oplus \cO_{X_2}\cdot T_1^2)= X_2 \cup \PP^1_{Z(I)}.$$
The complement is not necessarily affine, as $I$ is not necessarily principal. However, recalling the sheaf $F_I$ from {\sc Step 2b},  the morphism $(B_I)^\sst_{0} \to \PP_{X_2}(F_I)$ is an open embedding, whose image is the complement of the zero section. So  $(B_I)^\sst_{0}$ is the total space of the invertible sheaf $I\cO_{X_1}$ on $X_1$. Thus, $(B_I)^\sst_{0} \sslash \GG_m =X_1$ and the morphism $(B_I)^\sst_{0}\to X_1$ is smooth.

{\sc Step 3a: \em resolving $(B_I,D_{B_I})$.}
Let $D_{B_I}\subset B_I$ be the preimage of $D_2$. Applying resolution of pairs to $(B_I,D_{B_I})$ we obtain a functorial projective $\GG_m$-equivariant morphism $B\to B_I$ such that $B$ is regular and the preimage $D_{B}\subset B$ of $D_2$ is a simple normal crossings divisor. Here we use Theorem~\ref{resolth} if the characteristic is zero. In positive and mixed characteristic we may use parts (1) and (2) of  Hypothetical Statement \ref{Hyp:resolution} since $\dim B = \dim X_2+1$. In addition, $B\to B_I$ is projectively the identity outside of the union of $D_{B_I}$ and the singular locus of $B_I$, which is included in the preimage of $\PP_{X_2}((E_I)_{1}) = \PP_{X_2}(\cO_{X_2}\cdot U_1T_0 \oplus I\cdot U_0T_1)$. It follows that $(B,D_B)$ is a regular birational cobordism for $\phi$.

{\sc Step 3b: \em embedding.}
By the arguments of Section \ref{Sec:sequence-projective}, the composition $B\to B_I\to B_\cO$ is functorially a single blowing up of an ideal $J$. Write $\tilde J = J\cO_{B_I}$ so that $B=Bl_{\tilde J} B_I$. There is a functorially defined integer $d$ such that $\tilde J(d)$ is globally generated on $B_I$ relative to $X_2$. Using \cite[II.7.10(b)]{Hartshorne} we have an equivariant embedding of $B$ inside
%$$\PP_{X_2}(\widetilde E):=\PP_{X_2}\left(\Sym^d(E_I) \otimes \pi^{B_I/X_2}_*J(d)\right).$$
$$\PP_{X_2}(\widetilde E):=\PP_{X_2}\left(\pi^{B_I/X_2}_*\tilde J(d)\right).$$

We claim that $a_{\min}(B) = 0$ and $a_{\max}(B) = 2d$. First, since $E_I$ has weights $a_{\min}(E_I) = 0$ and $a_{\max}(E_I) = 2$ we have $a_{\min}(\Sym^d(E_I)) = 0$ and $a_{\max}(\Sym^d(E_I)) = 2d$. Second, the weights $0$ and $2d$ survive in the homogeneous coordinate ring of $B_I$ with respect to $O(d)$ as described in the steps above. Third, the weights in $\pi^{B_I/X_2}_*\tilde J(d)$  necessarily lie among those of $\Sym^d(E_I)$, so $a_{\min}(B) \geq 0$ and $a_{\max}(B) \leq 2d$. To show that the weights $0$ and $2d$ survive in $B$ it suffices to show this over a dense open set in $X_2$. Since $B\to B_I$ is projectively the identity over $U$, the weight $0$ and $2d$ components of $\pi^{B_I/X_2}_*\tilde J(d)$ are everywhere nonzero, as needed.

Inspecting the description of unstable loci in Section \ref{Sec:GIT-E}, Equation (\ref{Eq:unstable}) we note that $B^\sst_0 = B\times_{B_I}(B_I)^\sst_0$ and  $B^\sst_{2d} = B\times_{B_I}(B_I)^\sst_{2}$.

%Since $B\to B_I$ is projectively the identity\Michael{This is so only over $U$, see Step 3c.} on $(B_I)^\sst_{0}$ and $(B_I)^\sst_{2}$, we have that $\pi^{B_I/X_2}_*\tilde J(d)\hookrightarrow \Sym^d(E_I)$ is an isomorphism on the components of  $\GG_m$-weight $0$ or $2d$. Since $\Sym^d(E_I)$ has weights $0,\ldots,2d$ which survive in the ring of  $B_I$\Michael{Is it true that all of them survive? Certainly, 0 and $2d$ survive, and this is enough for us.} it follows that $a_{\min}(B) = 0, a_{\max}(B) = 2d$. Inspecting the description of unstable loci in Section \ref{Sec:GIT-E}, Equation (\ref{Eq:unstable}) we note that $B^\sst_0 = B\times_{B_I}(B_I)^\sst_0$ and  $B^\sst_{2d} = B\times_{B_I}(B_I)^\sst_{2}$.

{\sc Step 3c: \em $B$ is a cobordism for $\phi$ that respects $U$.} We have shown in steps 2d and 2e that the morphisms $q_2\:(B_I)^\sst_{2}\to X_2$ and $q_1\:(B_I)^\sst_{0}\to X_1$ are smooth. Functoriality of resolution of pairs with respect to $q_i$ implies that, once restricted to $(B_I)^\sst_{2}$, respectively  $(B_I)^\sst_{0}$, the morphism $B\to B_I$ is the pullback of the resolution $X'_2\to X_2$ of $(X_2,D_2)$, respectively $X'_1\to X_1$ of $(X_1,D_1)$. It follows that $B\times_{B_I}(B_I)^\sst_{2}\sslash\GG_m=X'_2$ and $B\times_{B_I}(B_I)^\sst_{0}\sslash\GG_m=X'_1$ and hence $B$ is a cobordism for $\phi$. Also, we note that $B\cap \PP(\widetilde E_{0}) = X'_1$ and  $B\cap \PP(\widetilde E_{2d}) = X'_2$, so Assumption \ref{Ass:spread} applies.

To show that $B$ is compatible with $U$ it suffices to show that both $B\to B_I$ and $B_I\to B_\cO$ are projectively the identity over $U$. This is so for the blowing up $B_I\to\PP^1_{X_2}$ because $I+I_{\{0\}}$ is the unit ideal on $\PP^1_U$, and this is so for the resolution $B\to B_I$ because $\PP^1_U$ is regular and disjoint from the preimage of $D_2$.

\end{proof}

\section{Germs of complex analytic spaces}\label{App:germs}
In this section we use germs to extend the category of complex analytic spaces to include certain Stein compacts. This will be used later to establish a tight connection between the scheme theory and complex analytic geometry. In particular, this is needed to develop a relative GAGA theory.

\subsection{Semianalytic sets}
We follow the setup of Frisch \cite{Frisch}. A subset $X$ of an analytic space $\cX$ is  called {\em semianalytic} if its local germs belong to the minimal class of germs, stable under finite unions and complements, generated by inequalities of the form $f(x)<0$ for real analytic $f$, see \cite[p. 120]{Frisch}. It is called  {\em a Stein} if $X$ has a fundamental system of neighborhood of Stein subspaces of $\cX$, see \cite[p. 123]{Frisch}.

\subsection{The category of germs}
A {\em germ of a complex analytic space} (or, simply, a germ) is a pair $(\cX,X)$ consisting of an analytic space $\cX$ and a semianalytic subset $X\subset\cX$. We call $X$ the {\em support} of $(\cX,X)$ and we call $\cX$ a {\em representative} of $(\cX,X)$. Sometimes, we will use the shorter notation $X=(\cX,X)$.

A morphism $\phi\:(\cX,X)\to(\cY,Y)$ consists of a neighborhood $\cX'$ of $X$ and an analytic map $f\:\cX'\to\cY$ taking $X$ to $Y$. We say that $f$ is a {\em representative} of $\phi$. Note that a morphism $(\cX,X)\to(\cY,Y)$ is an isomorphism if it induces a bijection of $X$ and $Y$ and an isomorphism of their neighborhoods.

We identify an analytic space $X$ with the germ $(X,X)$. In particular, the category of analytic spaces becomes a full subcategory of the category of germs.

\subsection{The structure sheaf}
Given a germ $(\cX,X)$ we provide its support with the {\em structure sheaf} $\cO_X:=\cO_\cX|_X=i^*\cO_\cX$, where $i\:X\into\cX$ is the embedding. In particular, we obtain a functor $\cF\:(\cX,X)\mapsto(X,\cO_X)$ from the category of germs to the category of locally ringed spaces.

\begin{remark}
We do not aim to develop a complete theory of semianalytic germs, so we do not study the natural question whether $\cF$ is fully faithful.
%\Michael{Should we think about this sometimes later?}
\end{remark}

\subsection{Closed polydiscs and convergent power series}
Consider an analytic affine space $\cX=\AA_\CC^n$ with coordinates $t_1\.t_n$. For any tuple $r$ of numbers $r_1\.r_n\in[0,\infty)$, by the closed polydisc $D=D_r$ of radius $r$ we mean the subset of $\cX$ given by the inequalities $|t_i|\le r_i$. Note that $r_i$ can be zero. By $\CC\{t_1\.t_n\}^\dag_r$ we denote the ring of overconvergent series in $t_1\.t_n$ of radius $r$. It is a noetherian regular excellent ring of dimension $n$, see \cite[Theorem 102]{Matsumura}.

\begin{lemma}\label{Lem:polydisc}
Let $D=D_r\subset\cX=\AA_\CC^n$ be a polydisc and $A=\cO_\cX(D)=\Gamma(\cO_D)$. Then,

(i) $\CC\{t_1\.t_n\}^\dag_r\toisom A$.

(ii) $\Gamma(D,\cdot)$ induces an equivalence between the categories of coherent $\cO_D$-modules and finitely generated $A$-modules, and higher cohomology of coherent $\cO_D$-modules vanish.

(iii) For any $a\in D$ the ideal $m_a=(t_1-a_1\.t_n-a_n)\subset A$ is maximal, and any maximal ideal of $A$ is of this form.

(iv) The completion of $A$ along $m_a$ is $\CC[[t_1-a_1\.t_n-a_n]]$.
\end{lemma}
\begin{proof}
The first claim is a classical result of analysis of several complex variables. Assertion (ii) follows from the fact that $D$ is the intersection of open polydiscs containing it, and the latter are Stein spaces. Assertion (iv) follows easily from (iii), so we will only prove (iii).

For any $f\in A$ the quotient $$g_1=(f(t_1\.t_n)-f(a_1,t_2\.t_n))/(t_1-a_1)$$ lies in $A$, so $f=(t_1-a_1)g_1+f_1(t_2\.t_n)$ with $f_1=f(a_1,t_2\.t_n)$. Applying the same argument to $t_2$ and $f_1$, etc., we will obtain in the end a representation $f=f(a_1\.a_n)+\sum_{i=1}^n(t_i-a_i)g_i$. In particular, $A/m_a=\CC$ and hence $m_a$ is maximal.

Conversely, assume that $m\subset A$ is maximal. The norm $\|f\|=\max_{x\in D}|f(x)|$ on $A$ induces a norm on the field $\kappa=A/m$, hence the completion $K=\hat\kappa$ is a Banach $\CC$-field. Thus, $K=\CC$ by Gel'fand-Mazur theorem, and we obtain that $t_i-a_i\in m$ for some $a_i\in\CC$. Finally, $|a_i|\le r_i$ as otherwise $t_i-a_i\in A^\times$.
\end{proof}

\subsection{Classes of morphisms}
Let $\phi\:(\cY,Y)\to(\cX,X)$ be a morphism of germs. We say that $\phi$ is {\em without boundary} if there exists a representative $f\:\cY'\to\cX$ such that $Y=f^{-1}(X)$. Let $P$ be one of the following properties: smooth, open immersion, closed immersion. We say that $\phi$ is $P$ if it is without boundary and has a representative which is $P$. We say that $\phi$ is an {\em embedding of a subdomain} (resp. {\em quasi-smooth}) if it possesses a representative which is an open immersion (resp. smooth).

\begin{remark}
The above terminology is chosen to match its non-archimedean analogue as much as possible.
\end{remark}

\subsection{Affinoid germs}\label{Sec:affgerm}
A germ $X$ is called {\em affinoid} if it admits a closed immersion into a germ of the form $(\CC^n,D)$ where $D$ is a closed polydisc. Such a germ is controlled by the ring $\cO_X(X)$ very tightly.

\begin{lemma}\label{Lem:affgerm}
Assume that $X$ is an affinoid germ and let $A=\cO_X(X)$ and $f\:(X,\cO_X)\to Y=\Spec(A)$ the corresponding map of locally ringed spaces. Then,

(i) $A$ is a quotient of a ring $\CC\{t_1\.t_n\}^\dag_r$; in particular it is an excellent noetherian ring.

(ii) $\Gamma(X,\cdot)$ induces an equivalence between the categories of coherent $\cO_X$-modules and finitely generated $A$-modules, and higher cohomology of coherent $\cO_X$-modules vanish.

(iii) $f$ establishes a bijection between $X$ and the closed points of $Y$.

(iv) For any point $x\in X$ with $y=f(x)$ the homomorphism $\cO_{Y,y}\to\cO_{X,x}$ is regular and its completion $\hatcO_{Y,y}\to\hatcO_{X,x}$ is an isomorphism.
\end{lemma}
\begin{proof}
In the case of a closed polydisc the assertion was proved in Lemma~\ref{Lem:polydisc}. In general, we fix a closed embedding $i\:X\into D$ into a closed polydisc. So, $\cO_X$ becomes a coherent $\cO_D$-algebra such that the homomorphism $\phi\:\cO_D\to\cO_X$ is surjective, and then all assertions except the first half of (iv) follow easily from the case of a polydisc. For example, $\Gamma(X,\cO_X)$ is a quotient of $\Gamma(D,\cO_D)$ since $H^1(D,\Ker\phi)=0$, thereby proving (i).

The only new assertion is that $\phi\:\cO_{Y,y}\to\cO_{X,x}$ is regular. This follows from the facts that $\hatphi$ is an isomorphisms and the local ring $\cO_{Y,y}$ is excellent (since it is a localization of the excellent ring $A$).
\end{proof}

\section{The complex relative GAGA Theorem}\label{Sec:relGAGA}

\subsection{Statement of the theorem}
Let $(\cX,X)$ be an affinoid germ  as in Appendix \ref{App:germs} with ring of global analytic functions $A$, and $r\geq 0$ an integer. Set $\PP^r_X = \CC\PP^r \times X$ and endow it with a locally ringed space structure using the sheaf $\cO_{\PP^r_X} = \cO_{\PP^r_\cX}|_{\PP^r_X}$. We have a germ $(\PP^r_{\cX},\PP^r_X)$ and a morphism of locally ringed spaces $h:\PP^r_X \to \PP^r_A$. The aim of this appendix is to prove the following extension of Lemma \ref{Lem:affgerm}:

\begin{theorem}[Serre's Th\'eor\`eme 3]\label{Th:GAGA}
Let $(\cX,X)$ be an affinoid germ with ring of global analytic functions $A$, and $r\geq 0$ an integer. Then the pullback functor $h^*: \Coh(\PP^r_A) \to \Coh(\PP^r_X)$  is an equivalence which induces isomorphisms on cohomology groups.
\end{theorem}

Since $(\cX,X)$ is closed in $(\CC^n, D)$ it suffices to consider the case  $(\cX,X)=(\CC^n, D)$. So from now on we make this assumption, and write $A$ for the ring of holomorphic functions on $X=D$.

We follow the steps of Serre's original proof \cite[\S 3]{Serre-GAGA} in some detail, to alleviate our skepticism that this generalization might actually work.  See also \cite{Kedlaya-GAGAnotes}, which sketches Serre's proof. One difficulty is that we do not know if $D \times \CC^r$ is Stein in the sense of \cite{Frisch} or \cite{Grauert-Remmert-Steinspaces}. The  problem is that if $\{D_i\}$ are the open polydiscs containing $D$ then $\{D_i\times\CC^r\}$ do not form a {\em fundamental} family of neighborhoods of $D \times \CC^r$, while functions on $D \times \CC^r$ are only guaranteed to extend to some member of a fundamental family of neighborhoods.
This is circumvented in Lemma \ref{Lem:GAGA-basic-sheaves}, which is the only point where we differ from the original arguments.

%\subsection{Acyclicity of sheaves on affine space} We do not know if the germ $(\cX\times \CC^r, X \times \CC^r)$ is Stein in the sense of \cite{Frisch}. However the following convenient result will suffice:

%\begin{proposition}[Cartan's Theorems A and B]\label{Prop:affine-acyclic} Let $\cF$ be a coherent sheaf on $(\cX\times \CC^r, X \times \CC^r)$. Then $\cF$ is generated by global sections and acyclic.
%\end{proposition}

 \subsection{Cohomology}

 \begin{proposition}[Serre's Th\'eor\`eme 1]
Let $\cF$ be a coherent sheaf on $\PP^r_A$. The homomorphism $h^*: H^i(\PP^r_A,\cF)\to H^i(\PP^r_D,h^*\cF)$ is an isomorphism.
 \end{proposition}

 \begin{lemma}\label{Lem:GAGA-basic-sheaves}
\begin{enumerate}\item We have $H^i(\PP^r_A,\cF)= H^i(\PP^r_D,h^*\cF)=0$ for $i>r$ and all $\cF$.
\item The proposition holds for $\cF = \cO_{\PP^r_A}$ for all $r\geq 0$.
\end{enumerate}
 \end{lemma}
 \begin{proof}
(1) For $H^i(\PP^r_A,\cF)=0$  use the standard \c{C}ech covering of $\PP^r_A$, which has only $r+1$ elements. We need to show $H^i(\PP^r_D,h^*\cF)=0$.

On the analytic side we mimic the standard argument for vanishing  using \c{C}ech cocycles of a covering by closed polydiscs instead of affine spaces.
%\Dan{Doesn't Lemma \ref{Lem:acyclicity} address this?}
Let $h^*\cF \to S^\bullet$ be the standard flabby resolution of $h^*\cF$ by discontinuous sections, so $H^i(Y, h^*\cF|_Y) = H^i(\Gamma(Y,S^\bullet))$ for any subset $Y\subset \PP^r_D$. Let $\CC^r\simeq U_i \subset \CC\PP^r$ be the standard open sets and let $D_i\subset U_i$ be the standard closed polydisc of fixed radius $>1$. Set $X_i = D \times D_i \subset \PP^r_D$ and for each subset $I \subset \{0\.n\}$  let $X_I = \cap_{i\in I} X_{i}$. Then $X_I$ are complex affinoids for $I\neq\emptyset$, hence  $H^i(X_I, h^*\cF|_{X_I})  =0=H^i(\Gamma(X_I,S^\bullet))$ for $i>0$ and $I\neq\emptyset$.

On the other hand $$\cC^\bullet(\{X_i\}, S^j)\ \  = \ \ \ \left[\oplus_{|I|=1} S^j_{X_I} \to   \oplus_{|I|=2} S^j_{X_I} \to\cdots\right]$$ is a flabby resolution of $S^j$ so  $H^0( \Gamma(\PP^r_D,\cC^\bullet(\{X_i\}, S^j))) = \Gamma(\PP^r_D,S^j)$ and for $i>0$ we have $H^i(\Gamma(\PP^r_D,\cC^\bullet(\{X_i\},S^j))) = 0$.

Consider the double complex $C^{p,q} = \oplus_{|I|=p}\Gamma(X_I,S^q)$ and its two edges $\Gamma(\PP^r,S^\bullet)$ and  $\check{C}^p = \oplus _{|I|=p}\Gamma(X_I,h^*\cF)$. We obtain that $$H^i(\PP^r_D,h^*\cF) = H^i(\Gamma(\PP^r,S^\bullet)) = \HH^i(C^{\bullet,\bullet}) = H^i(\check{C}^\bullet).$$ The latter is trivial in degrees $>r$.

(2) We have that $\Gamma(\cO_{\PP^r_A}) = A$ and $H^i(\cO_{\PP^r_A}) = 0$ for $i>0$ by \cite[Theorem III.5.1]{Hartshorne}.
%We also have that   $\Gamma(\cO_{\PP^r_D})=A$ by the maximum principle. For $i>0$ i
It suffices to show that $\pi_*\cO_{\PP^r_D} = \cO_D$ and $R^i\pi_*\cO_{\PP^r_D} = 0$ for $i>0$  where $\pi: \PP^r_D \to D$ is the projection, since $D$ is Stein. For this note that $\cO_{\PP^r_D} = j_r^{-1} \cO_{\PP^r_{\CC\PP^n}}$, where $j_r: \PP^r_{D }\to \PP^r_{\CC\PP^n}$ is the inclusion:
$$\xymatrix{\PP^r_{D}\ar[d]_\pi\ar[r]^{j_r}& \PP^r_{\CC\PP^n}\ar[d]^{\varpi}\\ D\ar[r]_{j_0}& \CC\PP^n.}$$ By the topological proper push-forward theorem \cite[Corollary VII.1.5]{Iversen} we have $$R^i\pi_*\cO_{\PP^r_D}  = j_0^{-1} R^i\varpi_*\cO_{\PP^r_{\CC\PP^n}},$$ and the result follows from Serre's original GAGA theorems.
 \end{proof}

\begin{lemma}\label{allrlem}
The proposition holds for $\cF = \cO_{\PP^r_A}(n)$ for all $r\geq 0$ and all integers $n$.
\end{lemma}
\begin{proof} Induction identical to \cite[section 13 Lemme 5]{Serre-GAGA}: the result holds for $r=0$ since $D$ is Stein. Supposing it holds for $r-1$ and all $n$, we have the exact sequence $0 \to \cO_{\PP^r_D}(n-1) \to   \cO_{\PP^r_D}(n) \to \cO_{\PP^{r-1}_D}(n) \to 0$ and the corresponding sequence for $\PP^r_A$. We obtain a canonical homomorphism of long exact sequences
{\small $$\xymatrix{ H^{i-1}(\PP^{r-1}_A,\cO(n)) \ar[r]\ar[d] &H^i(\PP^r_A,\cO(n-1))\ar[r]\ar[d] &H^{i}(\PP^{r}_A,\cO(n))\ar[r]\ar[d] &H^{i}(\PP^{r-1}_A,\cO(n))\ar[d] \\
H^{i-1}(\PP^{r-1}_D,\cO(n)) \ar[r] &H^i(\PP^r_D,\cO(n-1))\ar[r] &H^{i}(\PP^{r}_D,\cO(n))\ar[r] &H^{i}(\PP^{r-1}_D,\cO(n)).
}$$}
The vertical arrows on the right and left are isomorphisms by the inductive assumption. It follows that the result holds for $r$ and $\cO(n-1)$ if and only if it holds for $\cO(n)$. Since we have proven that it holds for $\cO$, it holds for all $n$.
\end{proof}

\begin{proof}[Proof of the proposition]
The proof is identical to Serre's Th\'eor\`eme 1. We apply descending induction on $i$ for all coherent $\PP^r_A$ modules $\cF$. The case of $i>r$ is proved by the lemma. Since $\cF$ is coherent there is an epimorphism  $\cE\to \cF$ with $\cE = \oplus_{i=1}^m \cO_{\PP^r_A}(-k_i)$. Denoting by $\cG$ the kernel, $\cG$ is coherent and we have a short exact sequence $$0\to \cG \to \cE\to \cF\to 0.$$  Since the map $h$ is flat we have an exact sequence $$0\to h^*\cG \to h^*\cE\to h^*\cF\to 0.$$

In the commutative diagram of cohomologies with exact rows
$$\xymatrix{ H^i(\PP^r_A,\cE) \ar[r]\ar[d] &H^i(\PP^r_A,\cF)\ar[r]\ar[d] &H^{i+1}(\PP^r_A,\cG)\ar[r]\ar[d] &H^{i+1}(\PP^r_A,\cE)\ar[d] \\
H^i(\PP^r_D,h^*\cE) \ar[r]&H^i(\PP^r_D,h^*\cF)\ar[r] &H^{i+1}(\PP^r_D,h^*\cG)\ar[r] &H^{i+1}(\PP^r_D,h^*\cE)
}$$
the vertical arrows on the left and right are isomorphisms by Lemma~\ref{allrlem}. By the induction hypothesis $H^{i+1}(\PP^r_A,\cG)\to H^{i+1}(\PP^r_D,h^*\cG)$ is an isomorphism as well. By the five lemma the result holds for $H^i(\PP^r_A,\cF)\to H^i(\PP^r_D,h^*\cF)$  as required.
\end{proof}

\subsection{Homomorphisms}

\begin{proposition}[Serre's Th\'eor\`eme 2]
For any coherent $\PP^r_A$-modules $\cF, \cG$ the natural homomorphism $$\Hom_{\PP^r_A} ( \cF,\cG) \to \Hom _{\PP^r_D} ( h^*\cF,h^*\cG)$$ is an isomorphism. In particular the functor $h^*$ is fully faithful.
\end{proposition}

\begin{lemma}
The sheaf homomorphism $$h^*\cHom_{\PP^r_A} ( \cF,\cG) \to \cHom _{\PP^r_D} ( h^*\cF,h^*\cG)$$ is an isomorphism.
\end{lemma}
\begin{proof}
This follows since $\cO_{\PP^r_D}$ is a flat $\cO_{\PP^r_A}$-module. Indeed, for a closed point $x\in\PP^r_D$ corresponding to a point  $x'=h(x)\in \PP^r_A$ we have
\begin{align*}\left(h^*\cHom_{\PP^r_A} ( \cF,\cG)\right)_x &= Hom_{\cO_{x'}} ( \cF_{x'},\cG_{x'})\otimes_{\cO_{x'}}\cO_{x}\\& =  Hom_{\cO_x} ( \cF_{x'}\otimes_{\cO_{x'}}\cO_{x},\cG_{x'}\otimes_{\cO_{x'}}\cO_{x})\\& = \cHom_{\PP^r_D} ( h^*\cF,h^*\cG)_x.
\end{align*}
\end{proof}

\begin{proof}[Proof of the proposition]
By Serre's Th\'eor\`eme 1,  $h^*$ preserves cohomology of coherent sheaves. Taking $H^0$ in the lemma the result follows.
\end{proof}

\subsection{The equivalence} It remains to show:
\begin{proposition} The functor $h^*$ is essentially surjective.
\end{proposition}
\begin{proof} This is an inductive argument on $r$ identical to Serre's Th\'eor\`eme 3 which we repeat below. The case $r=0$ follows from Lemma \ref{Lem:affgerm}. Assume the result is known for $r-1$ and let $\cF$ be a coherent sheaf on $\PP^r_D$. By Lemma \ref{Lem:global-generation} below there is an epimorphism $\phi:\cO(-n_0)^{k_0} \to \cF$, and applying this again to $\Ker (\phi)$ we get a resolution $\cO(-n_1)^{k_1} \stackrel{\psi}\to\cO(-n_0)^{k_0} \to \cF \to 0$. By Serre's Th\'eor\`eme 2 the homomorphism $\psi$ is the analytification of an algebraic sheaf homomorphism $\psi'$, so the cokernel $\cF$ of $\psi$ is also the analytification of the cokernel of $\psi'$.
\end{proof}

\begin{lemma}\label{Lem:global-generation} Assume the proposition holds for $r-1$. Then for any coherent sheaf $\cF$ on $\PP^r_D$ there is $n_0$ so that $\cF(n)$ is globally generated whenever $n>n_0$.
\end{lemma}
\begin{proof} By compactness it suffices to show that global sections of $\cF(n)$ generate $\cF(n)_x$ for fixed $x$. By Nakayama it suffices to show  that global sections of $\cF(n)$ generate the fiber $\cF(n)_x\otimes_{\cO_{D,x}} \CC_x$.

Picking a hyperplane $ \PP^{r-1}_D\simeq H\ni x$ we obtain an exact sequence $0 \to \cO(-1) \to \cO \to \cO_{H} \to 0$, giving an exact sequence $\cF(-1) \stackrel{\varphi_1}\to \cF \stackrel{\varphi_0}\to \cF_H \to 0$. Writing $\cP$ for $\Ker(\varphi_0) = Im(\varphi_1)$ we have two exact sequences  $$0 \to \cG \to\cF(-1) \to \cP\to 0 \qquad\text{and}\qquad 0 \to \cP \to\cF \to \cF_H \to 0,$$ noting that $\cG$ and $\cF_H$  are coherent sheaves on $H$. Twisting by $\cO(n)$  gives
$$  0 \to \cG(n) \to\cF(n-1) \to \cP(n) \to 0 $$ and $$0\to \cP(n) \to \cF(n) \to \cF_H(n) \to 0.$$
The long exact cohomology sequence gives
$$ H^1(\PP^r_D, \cF(n-1)) \to H^1(\PP^r_D, \cP(n)) \to H^2(H, \cG(n))$$ and
$$ H^1(\PP^r_D, \cP(n)) \to H^1(\PP^r_D, \cF(n)) \to H^1(H, \cF_H(n)).$$
By the assumption $\cF_H$ and $\cG$ are analytifications of algebraic sheaves, so for large $n$ the terms on the right vanish by Serre's Th\'eor\`eme 1. It follows that   $\dim H^1(\PP^r_D, \cF(n))$ stabilizes  for large $n$, and when it does the exact sequences above imply that $H^1(\PP^r_D, \cP(n)) \to H^1(\PP^r_D, \cF(n))$ is bijective so $H^0(\PP^r_D, \cF(n)) \to H^0(H, \cF_H(n))$ is surjective. Since the result holds for analytifications of algebraic sheaves, $\cF_H(n)$ is globally generated  for large $n$, implying that $\cF(n)_x\otimes_{\cO_{D,x}} \CC_x$ is generated by global sections, as needed.
\end{proof}

\bibliographystyle{amsalpha}
\bibliography{factor-qe}

\end{document}